\documentclass[11pt]{article}

\usepackage{times}
\usepackage{amsmath,amsfonts,amstext,amssymb,amsbsy,amsopn,amsthm,eucal}
\usepackage{txfonts}
\usepackage{dsfont}
\usepackage{graphicx}   

\newcommand{\Ric}{\text{Ric}}
\newcommand{\Vol}{\text{Vol}}

\newcommand{\NN}{\mathds{N}}
\newcommand{\RR}{\mathds{R}}
\newcommand{\Sn}{\mathds{S}}
\newcommand{\Hess}{\text{Hess}}

\newcommand{\Cl}{\text{Cl}}

\newcommand{\cS}{\mathcal{S}}
\newcommand{\cR}{\mathcal{R}}
\newcommand{\cH}{\mathcal{H}}

\begin{document}

\newtheorem{theorem}{Theorem}[section]

\newtheorem{proposition}[theorem]{Proposition}

\newtheorem{lemma}[theorem]{Lemma}

\newtheorem{corollary}[theorem]{Corollary}

\theoremstyle{definition}
\newtheorem{definition}[theorem]{Definition}

\theoremstyle{remark}
\newtheorem{remark}{Remark}[section]

\theoremstyle{remark}
\newtheorem{example}{Example}[section]

\theoremstyle{remark}
\newtheorem{note}{Note}[section]

\theoremstyle{remark}
\newtheorem{question}{Question}[section]

\theoremstyle{remark}
\newtheorem{conjecture}{Conjecture}[section]

\title{Sharp H\"older continuity of tangent cones for spaces with a lower Ricci curvature bound and applications}

\author{Tobias Holck Colding and Aaron Naber\thanks{Department of Mathematics, Massachusetts Institute of Technology, Cambridge, MA 02139.  Emails: colding@math.mit.edu and anaber@math.mit.edu.
         The  first author
was partially supported by NSF Grant DMS  0606629  and NSF FRG grant DMS
 0854774 and the second author by an NSF Postdoctoral Fellowship.}}


\date{\today}
\maketitle
\begin{abstract}
We prove a new kind of estimate that holds on any manifold with a lower Ricci bound.   It relates the geometry of two small balls with the same radius, potentially far apart, but centered in the interior of a common minimizing geodesic.   It reveals new, previously unknown, properties that all generalized spaces with a lower Ricci curvature bound must have and it has a number of applications.

This new kind of estimate asserts that the geometry of small balls along any minimizing geodesic changes in a H\"older continuous way with a constant depending on  the lower bound for the Ricci curvature, the dimension of the manifold, and the distance to the end points of the geodesic.  We give examples that show that the H\"older exponent, along with essentially all the other consequences that we show follow from this estimate, are sharp.  The unified theme for all of these applications is convexity.

Among the applications is that the regular set is convex for any non-collapsed limit of Einstein metrics.  In the general case of potentially collapsed limits of manifolds with just a lower Ricci curvature bound we show that the regular set is weakly convex and $a.e.$ convex, that is almost every pair of points can be connected by a minimizing geodesic whose interior is contained in the regular set.  We also show two conjectures of Cheeger-Colding.   One of these asserts that the isometry group of any, even collapsed, limit of manifolds with a uniform lower Ricci curvature bound is a Lie group; the key point for this is to rule out small subgroups.  The other asserts that the dimension of any limit space is the same everywhere.    Finally, we show a Reifenberg type property holds for collapsed limits and discuss why this indicates further regularity of manifolds and spaces with Ricci curvature bounds.
\end{abstract}

\section{Introduction}

We begin by giving two almost equivalent versions of the main H\"older continuity result for tangent cones.   After having done this we discuss and prove various of the almost immediate consequences.   In the final subsection of this introduction we describe the examples that show that almost all of these results are sharp, including the H\"older behavior in the main theorem.

\subsection{H\"older continuity of tangent cones}

Let $M$ be a complete $n$-dimensional manifold with
\begin{equation}
\Ric_M\geq -(n-1)\, ,  \label{e:elowerRicb}
\end{equation}
and suppose that $\gamma:[0,\ell]\to M$ is a (unit speed) minimizing geodesic.  Our main theorem is the following result that has two essentially equivalent formulations (the second formulation concerns limit spaces and will be given shortly):

\begin{theorem}   \label{t:holder}
(H\"older continuity of geometry of small balls with same radius).
There exists $\alpha(n)$, $C(n)$ and $r_0(n)>0$ such that given any $\delta>0$ with $0<r<r_0\delta\ell$ and $\delta\ell<s<t<\ell-\delta\ell$, then
\begin{equation}
d_{GH}(B_r(\gamma (s)),B_r(\gamma (t)))<\frac{C}{\delta\ell}\, r\,|s-t|^{\alpha(n)}\, .
\end{equation}
\end{theorem}

We will see from the proof that $\alpha(n)$ is \textit{effectively} $\frac{1}{2}$, which is to say that the Gromov-Hausdorff approximation (or map)  is in fact $\frac{1}{2}$-H\"{o}lder continuous on sets of arbitrarily full measure.  Note that $\alpha$, $C$ and $r_0$ in this Theorem do not depend on $\gamma$ or even $M$.

\vskip2mm
In fact this theorem, as everything else in this paper, holds for possibly singular limits of manifolds.  To state the Theorem for singular limits let us consider a sequence $M_i^n$ of $n$-dimensional manifolds (possibly collapsing) each satisfying (\ref{e:elowerRicb}) and let $M_{\infty}$ be a Gromov-Hausdorff limit of $M_i$.    (So $M_{\infty}$ may have lower Hausdorff dimension.)  We say a geodesic
$$\gamma:[0,\ell]\rightarrow M_\infty$$
is a limit geodesic if there exists geodesics $\gamma_i:[0,\ell_i]\rightarrow M_i$ with $\ell_i\rightarrow\ell$ such that
$$\gamma_i\rightarrow\gamma\, ,$$
pointwise.  Though it is not clear a geodesic on $M_\infty$ is always a limit geodesic, what will be most important for us is that given any two points $x,y\in M_\infty$ there always exists a limit geodesic connecting them.  In fact we will see that the collection of limit geodesics on $M_\infty$ are in abundance and have all the measure theoretic properties one would hope for (see Section \ref{ss:Convexity1}, Section \ref{ss:Convexity2} and Appendix \ref{s:extendgeodesic}).  The main estimate for singular spaces is then:

\begin{theorem}\label{t:holderlimit}
Theorem \ref{t:holder} holds for limit geodesics in $M_{\infty}$.
\end{theorem}

If $(M_{\infty},d_{\infty})$ is a Gromov-Hausdorff limit, $x\in M_{\infty}$ and $s_j\to 0$, then a subsequence of the blows up $(M_{\infty},s_j^{-1}\,d_{\infty},x)$ converges
in the Gromov-Hausdorff topology to a metric space called a
tangent cone at $x$.  If the sequence $M_i$ is non-collapsing then it was shown in \cite{ChC2} that tangent cones are metric cones; however, this is not necessarily the case in the collapsing case, see example 8.95 of \cite{ChC2}, and in fact tangent cones may not even be polar spaces in the collapsed case, see \cite{M4}.  Even for general non-collapsed limits tangent cones can be non-unique, see
\cite{ChC2} for examples and cf. \cite{P2}\footnote{In the Einstein setting
uniqueness of tangent cones is unknown.}.

\vskip2mm
Theorem \ref{t:holder} implies that tangent cones change in a H\"older continuous way even for collapsed limits.    Note that this has to be understood in an appropriate way because of the non-uniqueness of the tangent cones.  To do this suppose that $x,y\in M_\infty$ are points in $M_\infty$ with $Y_x$, $Y_y$ tangent cones in $M_\infty$ centered at $x$ and $y$, respectively.  Then we say $Y_x$ and $Y_y$ \textit{come from the same sequence of rescalings} if there exists a sequence $s_j\to 0$ such that
\begin{align}
(M_{\infty},s^{-1}_jd_{\infty},x)\to Y_x\, ,\notag\\
(M_{\infty},s^{-1}_jd_{\infty},y)\to Y_y\, .
\end{align}

Now given a limit geodesic $\gamma:[0,\ell]\rightarrow M_\infty$ we state the main estimate for tangent cones along $\gamma$.

\begin{theorem}\label{t:tangentholder}
(H\"older continuity of tangent cones).
There exists $\alpha(n)$ and $C(n)>0$ such that given any $\delta>0$ with $\delta\ell<s<t<\ell-\delta\ell$ and $Y_{\gamma(s)}$, $Y_{\gamma(t)}$ tangent cones from the same sequence of rescalings, then we have
\begin{equation}
d_{GH}(B_1Y_{\gamma (s)},B_1Y_{\gamma (t)})<\frac{C}{\delta\ell}\,|s-t|^{\alpha(n)}\, .
\end{equation}
Here $B_1Y_s$ and $B_1Y_t$ are the unit balls around the `cone' tips.
\end{theorem}
\begin{remark}
An immediately corollary is that tangent cones from the same sequence of rescalings on $\gamma$ change continuously, see also Example \ref{ex:Generalized_Trumpet}.
\end{remark}

This Theorem and Examples \ref{ss:Example4}, \ref{ss:Example5} should be contrasted to a result of Petrunin, \cite{Pn}, who showed a conjecture of Yu. Burago asserting that for Alexandrov spaces the tangent cones remains the same
along the interior of a geodesic.  Since the regular set of an Alexandrov space is the collection of points whose cone is Euclidean space it follows easily from Petrunin's result that for an Alexandrov space the regular set is convex.

 A useful consequence of the H\"older continuity of the tangent cones that we will use several times is that the set of interior points of a geodesic where the tangent cone is unique and equal to a given metric space is closed relative to the interior.    This is the following:

\begin{corollary}\label{c:closed}
If $\gamma:[0,\ell]\to M_{\infty}$ is a limit geodesic and $(Y,0)$ is a fixed pointed metric space, then
\begin{enumerate}
\item The set of interior points on $\gamma$ where the tangent cone is unique and equal to $Y$ is closed relative to the interior.    \label{e:closed1}
\item If the set is also dense in the interior, then it is all of the interior.
\end{enumerate}
\end{corollary}
\begin{remark}
In fact the assumption of uniqueness of the tangent cones is not necessary.  Relative to any sequence $r_j\rightarrow 0$ the collection of points of $\gamma$ whose tangent cone from this sequence of rescalings is equal to $Y$ is a closed set.
\end{remark}

\subsection{Convexity of the regular set for non-collapsed Einstein limits}
Let $M_{\infty}$ be a pointed limit of $n$-dimensional manifolds $M_i$ with $p_i\in M_i$ and
\begin{align}
\Vol (B_1(p_i))&\geq v>0\, ,\\
|\Ric_{M_i}|&\leq (n-1)\, .
\end{align}

As mentioned above for non-collapsing limits it was shown in \cite{ChC2} (see theorem 5.2 there) that tangent cones are metric cones of Hausdorff dimension $n$.  A metric cone $C(X)$ with
cross-section $X$ is a warped product $(0,\infty)\times_r X$ with
warping function $f(r)=r$, $r\in (0,\infty)$.  By theorem 5.2 of \cite{ChC2} $X$ is a length space
with diameter at most $\pi$ and dimension equal to $(n-1)$.  The regular set of
$M_{\infty}$ are points where a neighborhood is a smooth Einstein manifold;
the complement of the regular set is the singular set.  It can be
shown, see theorem 0.3 in \cite{C3} and cf. \cite{ChC2}, that a point is regular if and only if \underline{one} of its tangent cones is isometric to $\RR^n$; so in this case uniqueness follows.  As a first application of Theorem \ref{t:tangentholder} we have the following convexity result for the regular set.

\begin{theorem}  \label{t:Regbounded}
(Convexity of the regular set).
The regular set is convex for any non-collapsed limit of a sequence of $n$-manifolds with uniformly bounded Ricci curvatures.

That is, if $M_i$ are as above and $M_{\infty}$ is a limit of the $M_i$'s with $\gamma_{\infty}:[0,\ell]\to M_{\infty}$ a limit geodesic segment in $M_{\infty}$ such that one point on $\gamma_{\infty}$ is regular (possibly an endpoint), then every interior point of $\gamma_{\infty}$ is a regular point.
\end{theorem}


Note that in section 3 of \cite{ChC3} a much weaker statement is shown.  Namely, there it is shown (see, in particular, corollary 3.10 in \cite{ChC3}) that in a non-collapsed limit of spaces with a uniform lower Ricci curvature bound any pair of regular points can be connected by a curve consisting entirely of almost regular points.  See \cite{ChC3} for the precise statement.   See also the three last paragraphs on page 408 of \cite{ChC3} where it is discussed that one would like to know that $\cR$ is connected in the collapsed case.

\vskip6mm
As a consequence of \ref{e:closed1}) in Corollary \ref{c:closed} we get Theorem \ref{t:Regbounded}:

\begin{proof}
(of Theorem \ref{t:Regbounded}).
By \cite{C3}, see also section 7 of \cite{ChC2}, the regular set of a non-collapsed limit of spaces with uniformly bounded Ricci curvatures is an open set.  This follows since by \cite{C3} the following two are equivalent for such a limit:
\begin{enumerate}
\item A tangent cone at $x$ is $\RR^n$.
\item An open neighborhood of $x$ is a $C^{1,\beta}$ Riemannian manifold.
\end{enumerate}
By Corollary \ref{c:closed}.1 the regular points in $\gamma_\infty$ are also a closed set, hence the Theorem follows.
The theorem follows from this together with \ref{e:closed1}) of Corollary \ref{c:closed}.
\end{proof}

The following effective version of the regular set being convex, or rather that if one endpoint of a limit minimizing geodesic is a regular point, then the whole interior consists of regular points.  This is the following which is interesting even when $M_i$ and $M_{\infty}$ are all the same smooth Riemannian manifold:

\begin{theorem}  \label{t:blowup}
(Rate of blow-ups along geodesics). There exists $c(n),r_0(n)>0$ and $\alpha(n)>1$ such that if $M_{\infty}$ is as in Theorem \ref{t:Regbounded} with $\gamma_{\infty}:[0,2\ell]\to M_{\infty}$ a limit minimizing unit speed geodesic with $\ell\leq 1$, $r\leq r_0\ell$, and
\begin{equation}
d_{GH}(B_r(\gamma_{\infty}(\ell)),B_r(0))<\epsilon\,r\, ,
\end{equation}
then for all $0<s<1$
\begin{equation}  \label{e:rateofblowup}
d_{GH}(B_{cs^{\alpha}r}(\gamma_{\infty}(s\ell)),B_{cs^{\alpha}r}(0))<\epsilon\,c\,s^{\alpha}\,r\, .
\end{equation}
Here $B_t(0)\subset \RR^n$ is the Euclidean ball.

Moreover, if all $M_i$'s are Einstein, then the curvature $R$ blows-up at $\gamma (0)$ at most to the power $-2\alpha$.  That is, after choosing $c$ even smaller we get by combining (\ref{e:rateofblowup}) with \cite{C3} that all balls $B_{cs^{\alpha}r}(\gamma(s))$ are smooth and
\begin{equation}
\sup_{B_{cs^{\alpha}r}(\gamma(s))}|R|\leq s^{-2\alpha}\,r^{-2}\, .
\end{equation}
\end{theorem}

Recall that a ball $B_r(p)$ in a manifold $M$ is said to have {\it bounded geometry} if the (sectional) curvature $R$ is bounded by $r^{-2}$ and the injectivity radius at $p$ is at least $r$.  The next corollary is for simplicity only stated for Ricci flat manifolds though holds with obvious changes for Einstein manifolds.   Roughly speaking it says that in an Einstein manifold regions with bounded geometry propagates throughout the manifold (as any pair of points in the manifold obviously can be joined by a minimizing geodesic and thus bounded geometry near one point mean by the corollary bounded geometry near the other).

\begin{corollary}      \label{c:ricciflatblowup}
Given an integer $n$, there exist constants $\alpha=\alpha (n)\geq 1$ and $C=C(n)>0$ such that the following holds:
Suppose that $M^n$ is a Ricci flat $n$-manifold and $\gamma:[0,2L]\to M$ is a unit speed minimizing geodesic, then the radius of balls centered at $\gamma (r)$ that have  bounded geometry decay at most like
\begin{equation}
C\,\left(\frac{r}{L}\right)^{-\alpha}
\end{equation}
from $\gamma (L)$ to $\gamma (0)$.
\end{corollary}

Theorem \ref{t:blowup} follows from iterating the following lemma:

\begin{lemma}
Given an integer $n$, there exists $\epsilon=\epsilon (n)$, $\delta=\delta (n)$, $r_0=r_0(n)>0$, so that if $M^n$ is a smooth manifold with $|\Ric|\leq \epsilon$ and $\gamma:[0,2]\to M$ is a unit speed minimizing geodesic with $r\leq r_0$ and
\begin{equation}
d_{GH}(B_r(\gamma_{\infty}(1)),B_r(0))<\epsilon\,r\, ,
\end{equation}
then
\begin{equation}
d_{GH}\left(B_{\delta r}\left(\gamma_{\infty}\left(1/2\right)\right),B_{\delta r}(0)\right)
<\epsilon\,\delta\,r\, .
\end{equation}
\end{lemma}

\begin{proof}
Theorem \ref{t:holder} above gives the assertion except with $2\epsilon$ instead of $\epsilon$ and at $\nu$ instead of at $\frac{1}{2}$,  where $\nu$ is close to $1$ depending on $\epsilon$, but  independent of $r$.  By \cite{C3} (see also section 7 of \cite{ChC2}) the metric on the ball $B_{\frac{r}{2}}(\gamma (\nu))$ is $C^{1,\beta}$ with fixed small scale invariant $C^{1,\beta}$ norm of the difference of the metric and the flat $g_{i,j}=\delta_{i,j}$ Euclidean metric on the ball, provided $\epsilon$ is fixed small.  Hence, going to a smaller scale gives $\epsilon$ as opposed to $2\epsilon$.    Repeating this argument $\frac{1}{2\nu}$ many times yields the claim.
\end{proof}

We note that there exists a limit $M_{\infty}$ of a non-collapsing sequence of smooth $4$-manifolds $M_i^4$ with $\Ric_{M_i}\geq 0$ and a unit speed geodesic $\gamma_{\infty}:[0,1]\to M_{\infty}$ such that $\gamma_{\infty} (1)$ is a smooth point, but the curvature blows up faster than quadratically at $\gamma_{\infty}(0)$.     This follows from one of the examples in \cite{CN1} that show that there is such a limit where one tangent cone at $\gamma_{\infty}(0)$ is smooth and another not; cf. also section 8 of \cite{ChC2}, \cite{P2}, and the examples section later in this paper.  In the Ricci flat case the Eguchi-Hanson metrics, \cite{EH}, show that quadratic blow up of the curvature (and thus linear blow up of the  geometry) is the best that one can hope for.    This would correspond to that $\alpha$ can be chosen to be $1$ in Corollary  \ref{c:ricciflatblowup}.    In fact, we conjecture that this is the case:

\begin{conjecture}
$\alpha$ in Corollary \ref{c:ricciflatblowup} can be chosen to be $1$.  (A similarly statement should hold for general Einstein manifolds).
\end{conjecture}

An affirmative answer to this conjecture would have various applications.  In particular, it would immediately give the following:

\begin{conjecture}
If $M^n$ is an open Ricci flat manifold with Euclidean volume growth and one tangent cone at infinity is smooth, then all tangent cones at infinity are smooth.  (A similarly statement should hold for local tangent cones of non-collapsed limits of Einstein manifolds.)
\end{conjecture}

For later use we conclude this subsection by mentioning that there is a natural
stratification of the singular set of $M_{\infty}$ based on
tangent cones; see top of page 410 of \cite{ChC2}.  This is valid even in the case of non-collapsed limit of manifolds with a uniform lower Ricci curvature bound.  Namely,
\begin{align}
 \cS_k\equiv\{x:x\text{ is singular and no tangent cone at }x\text{ splits off a }\RR^{k+1}\text{ factor}\}\, .
\end{align}
That is, no tangent cone at $x$ is isometric to $\RR^{k+1}\times Y$ for some metric space $Y$.  Thus
$$\cS_0\subset \cS_1\subset \cdots\subset \cS_n\footnote{It was shown in \cite{ChC2} that $\cS=\cS_{n-2}$ and that $\dim \cS_k\leq k$.  Hence, the conclusion that non-collapsed limits of Einstein manifolds are smooth outside a closed subset of Hausdorff codimension $2$.  See \cite{C4}, \cite{Ch}, \cite{F2}, \cite{Ga}, \cite{W} for surveys of these results.}.$$

\subsection{Branching geodesics and local dimension}
A geodesic is said to be branching if there exists another geodesic that coincide with $\gamma$ on a open subset, but  that at some point the two curves depart (branch) from each other.  Precisely, there does not
exists a common extension of the two geodesics.  Obviously, for smooth or even $C^{1,\beta}$ manifolds branching cannot occur as geodesics are entirely determined
by their initial conditions (initial velocity).  Even for Alexandrov spaces it follows directly from the Toponogov
triangle comparison theorem that geodesics do not branch; see page 384 of \cite{GvPe} and \cite{BGP}.
However, for general limits of manifolds with lower Ricci curvature bounds it is unknown whether or not
geodesics can branch in the interior; cf. \cite{CN2}.  Some simple branching that potentially could come from one-dimensional
pieces have been rule out in section 5 of \cite{ChC3}; see the example below.  Moreover, we will recall below the known examples of limit spaces that have geodesics that start out tangent and then branch.    An immediate corollary of the results above is:

\begin{corollary}
If $M_{\infty}$ is a non-collapsing limit of a sequence of $n$-manifolds with uniformly bounded Ricci curvature, and
$\gamma_{\infty}$ is a branching limit
geodesic, then $\gamma_{\infty}$ is entirely contained in the singular set of
$M_{\infty}$.
\end{corollary}

In \cite{ChC2} and \cite{ChC3} the following two examples of metric spaces were discussed.  It was show there that one of them could in fact occur as a limit spaces whereas the other could not.

\begin{example}\label{ex:1}
(The horn and the trumpet; see example 8.77 of \cite{ChC2} and example 5.5 of \cite{ChC3}).
As shown in \cite{ChC2}, the metric horn $Y^5$, with metric
\begin{equation}
dr^2+\left(\frac{1}{2}r^{1+\epsilon}\right)^2\,g^{\Sn^4}\, ,
\end{equation}
arises as the limit of a collapsing sequence, $(M^8_i,g_i)$.
The trumpet is the space obtained by attaching at the origin, a line
segment, $[-j, 0]$, to the horn, $Y^5$. It follows from theorem 5.1 of \cite{ChC3} (see example 5.5 in \cite{ChC3}), that for
no $j > 0$ does the trumpet arise as the limit of a sequence of manifolds with uniform lower Ricci curvature bounds.
\end{example}

For the horn all points have a unique tangent cone.  At the tip the tangent cone is equal to the half-line $[0,\infty)$ and all other points have tangent cone equal to $\RR^5$.  The trumpet has also unique tangent cones, but on the entire closed line segment tangent cones are equal to $\RR$.  In particular, for the trumpet there are two open (non-empty) subsets so that on one the tangent cone is unique and equal to one Euclidean space and on the other tangent cones are also unique, but equal to a different Euclidean space.  It was conjectured in \cite{ChC3} that this should not happen for limits of manifolds with a uniform lower Ricci curvature bound.  The fact that the trumpet could not occur as a limit was given as support of this conjecture, however the tools to show even simple generalizations of the trumpet (e.g. the trumpet cross a torus) cannot arise as a limit have remained elusive.  We prove this conjecture in full in the next subsection, but first let us apply Theorem \ref{t:tangentholder} to see that these generalized trumpets can not occur as limit spaces.  The proof of the full conjecture is in the same spirit, if technically more involved.

\begin{example}\label{ex:Generalized_Trumpet}
Let $Z^5\equiv [-j,0]\cup Y^5$ be the trumpet constructed in Example \ref{ex:1} and let $X\equiv Z^5\times M^n$ be the trumpet cross an $n$-manifold.  We claim $X$ cannot arise as a Gromov-Hausdorff limit of manifolds with lower Ricci bounds.  To prove this let $x,y\in X$ be points in $X$ such that the tangent cones in a neighborhood of $x$ are unique and equal to $\RR^{n+1}$, while the tangent cones in a neighborhood of $y$ are unique and equal to $\RR^{n+5}$.  Let $\gamma$ be a limit geodesic connecting $x$ and $y$ and note that the tangent cones at each point of $\gamma$ are unique and are isometric to either $\RR^{n+1}$ or $\RR^{n+4}$, with both arising as tangent cones at some interior point.  However as in the remark following Theorem \ref{t:tangentholder} the tangent cones are changing continuously, which is not possible if the tangent cones acquire both of only two possible values.
\end{example}

Horns are examples of length spaces where geodesics that are initially tangent branch and trumpets are examples of length spaces where geodesic branch at some interior point.

In \cite{CN2} we will construct almost Euclidean limit spaces where geodesics that are initially tangent branch; just like in the example of the horn, but with the additional property that the space is almost maximal.

\subsection{Convexity of the regular set in general limits and constant local dimension}\label{ss:Convexity1}

In this subsection we will state and prove a convexity result for general limits that follows from our main H\"older continuity result.  Once we have that we are in a position to prove a conjecture of Cheeger-Colding.  They conjectured that the dimension of any limit is the same at almost every point, see Example \ref{ex:Generalized_Trumpet} where we use Theorem \ref{t:tangentholder} to prove this in a simplified setting.  To make the general results precise we need to recall the renormalized limit measures and the measured Gromov-Hausdorff convergence.  This summarizes some of the results in section 1 of \cite{ChC2}.  These measures were first constructed by Fukaya, \cite{F1}, who used a different argument than the one given in \cite{ChC2}.

In the non-collapsed case, the limit measure exists without the necessity
of passing to a subsequence, or of renormalizing the measure.
The unique limit measure is just the Hausdorff measure, $\cH^n$; see theorem
5.9 in \cite{ChC2}. (If, for the sake of consistency, one does renormalize the measure,
then one obtains a multiple of $\cH^n$, where as usual, the normalization
factor depends on the choice of base point.) However, in the collapsed
case the renormalized limit measure on the limit space can depend on
the particular choice of subsequence; see example 1.24 in \cite{ChC2}.

Let $M^n$ satisfy $\Ric_{M^n}\geq -(n-1)$.   Fix $p$ and define the renormalized volume function by
\begin{equation}
\underline{V}(x,r) =\underline{\Vol} (B_r(x))=\frac{\Vol (B_r(x))}{\Vol (B_1(p))}\, .
\end{equation}

Let $(M_i^n,p_i,\underline{V}_i)$ be a pointed sequence of $n$-dimensional manifolds with reference points $p_i$ and renormalized measures $\underline{V}_i$ defined with respect to these reference points.   Combining the proof of Gromov's compactness theorem with a  modification of the proof of the theorem of Arzela-Ascoli, one
obtains that a subsequence of these metric-measure spaces converges to a metric-measure space.  The limit will automatically satisfy a local doubling condition both metrically and measure-wise (in fact, it satisfies a Bishop-Gromov volume inequality); see section 1 of \cite{ChC2} for details.
The Radon measure $\nu$ uniquely defined from such a limit
$\underline{V}_i\to \nu$ is said to be a renormalized limit measure on the Gromov-Hausdorff limit $M_{\infty}$.

Suppose that $M_{\infty}$ is a measured Gromov-Hausdorff limit with renormalized measure $\nu$.  If $A\subset M_{\infty}$ has renormalized measure zero, ie, $\nu (A)=0$, then for any pair of balls in $M_{\infty}$ almost all limit minimal geodesics from one ball to the other intersect $A$ in a set of measure zero.  Precisely, we have the following (which is a direct consequence of the segment inequality theorem 2.6 in \cite{ChC4}, see also theorem 2.11 in \cite{ChC1} - we are below using the notation from theorem 2.11).   Let $A_1$, $A_2$ be open subsets of $M_{\infty}$ with $\gamma_{a_1,a_2}: [0,\overline{a_1,a_2}]\to M_\infty$ a limit minimal geodesic connecting $a_1$ and $a_2$ and let
\begin{equation}  \label{e:segment}
I_{a_1,a_2}(A)=\inf_{\gamma_{a_1,a_2}}|t\,:\, \gamma_{a_1,a_2} (t)\in A|
\end{equation}
where $|\cdot|$ denotes the measure of a set and the infimum in (\ref{e:segment}) is taking over all minimal limit geodesics connecting $a_1$ and $a_2$.  Equip the product $A_1\times A_2$ with the product measure $\nu\times \nu$.  Then we have the following:

\begin{lemma}  \label{l:segment}
If $\nu(A)=0$ then for $\nu\times\nu$ almost every $(a_1,a_2)\in A_1\times A_2$ we have
\begin{equation}
|I_{a_1,a_2}(A)|=0\, .
\end{equation}
\end{lemma}

\begin{proof}
Apply theorem 2.6 of \cite{ChC4} to the indicator function of $A$.
\end{proof}

We will also need the notion of a regular point in the collapsed case.  We say that $x\in M_{\infty}$ is a $k$-regular point if every tangent cone is equal to $\RR^k$.   There are examples, see \cite{M2},  where one, but not all, of the tangent cones at a point is $\RR^k$, so uniqueness is a non-trivial assumption.\footnote{In the non-collapsing case if one tangent cone is $\RR^n$, then all are; see theorem 0.3 in \cite{C3}.}  We will write
\begin{align}
 \cR_k\equiv\{x:x\text{ is k-regular}\}\, ,
\end{align}
for the set of $k$-regular points and
\begin{align}
 \cR\equiv \cup_k \cR_k\, ,
\end{align}
for the set of all regular points.  The singular set $\cS$ is the complement of the regular set, so
\begin{align}
\cS=M_{\infty}\setminus \cR \, .
\end{align}
In theorem 2.1 of \cite{ChC2} it is shown that $\nu (\cS)=0$ (and hence $\cR$ is dense in $M_{\infty}$).    From this together with Theorem \ref{t:holder} and the segment inequality in form of Lemma \ref{l:segment} we conclude:

\begin{lemma}   \label{l:collapsedconvex}
For $\nu\times\nu$ almost every point $(a_1,a_2)\in A_1\times A_2$ there exists a limit minimal geodesic from $a_1$ to $a_2$ whose interior lies entirely inside $\cR_k$ for some $k$.  That is, the entire interior consists of  $k$-regular for the same $k$.
\end{lemma}

\begin{proof}
Applying Lemma \ref{l:segment} to $A=\cS$ and using that by theorem 2.1 of \cite{ChC2} $\nu (\cS)=0$ it follows that for a.e. $a_1$ and $a_2$ there is a limit minimal geodesic $\gamma:[0,\ell]\to M_{\infty}$ connecting them so that the intersection of $\gamma$ with the regular set $\cR$ has full measure (and hence is dense) in $[0,\ell]$.  For any such pair $(a_1,a_2)$ it follows now easily from Corollary \ref{c:closed} that if $0<s<t<\ell$ with $\gamma (s)\in \cR_{k_s}$, and $\gamma (t)\in \cR_{k_t}$, then $k_s=k_t$.

Namely, it follows from \ref{e:closed1}) of Corollary \ref{c:closed} that for each $k$ the intersection
\begin{align}
 \breve\gamma\cap\cR_k\, ,
\end{align}
of the interior of the geodesic is closed relative to the interior.  Since $k\leq n$ by \cite{C3} (see also \cite{ChC2}), there are at most finitely many $\cR_k$'s that are non-empty.  It follows that the intersection of $\cR=\cup_k\cR_k$ with the interior of $\gamma$ is closed relative to the interior and since it is also dense it follows that the regular set is all of the interior.  Since $\{\cR_k\}_k$ are all pairwise disjoint, the intersection of each $\cR_k$ with the interior of the geodesic is closed, and the union is all of the interior, it now follows that there is only one $k$ so that $\breve\gamma\cap\cR_k$ is non-empty.
\end{proof}

We have now that the dimension of the tangent cone is $\nu$ almost everywhere the same.  Results of this form were originally proved in the four dimensional Einstein case in the fundamental works of \cite{BKN} and \cite{Ti}.  In the general noncollapsed case the result was first proved in \cite{ChC2}, where the following collapsed version was conjectured to hold as well.

\begin{theorem}\label{t:uniquek}
There is a unique $k$ so that
\begin{equation}
\nu (\cR\setminus \cR_k)=0\, .
\end{equation}
Combining this with that $\nu (\cS)=0$ by theorem 2.1 of \cite{ChC2} it follows that $\nu (M_{\infty}\setminus \cR_k)=0$.  We call this $k$ the dimension of $M_\infty$.
\end{theorem}
\begin{remark}
It is not clear that the dimension of the regular set is equal to the Hausdorff dimension of the limit space.
\end{remark}
\begin{proof}
This is a consequence of Lemma \ref{l:collapsedconvex} together with the following technical result (see Corollary \ref{c:extending} below) that we prove in an appendix:
\begin{itemize}
\item a.e. pair $(x,y)\in M_{\infty}\times M_{\infty}$ is in the interior of a limit geodesic, that is, a.e. pair lie on a limit geodesic that can be extended as a limit geodesic on both sides.
\end{itemize}

To see the result now assume $\exists$ $k,l$ such that $\nu(\cR_k),\nu(\cR_l)>0$.  Then by the segment inequality and the above statement there exists a limit minimizing geodesic $\gamma_\infty$ which intersects both $\cR_k$ and $\cR_l$ in the interior while satisfying Lemma \ref{l:collapsedconvex}.  Hence $k=l$ as claimed.
\end{proof}

\subsection{$\cR_k$ is connected and weakly convex, and the isometry group of a limit is a Lie group}\label{ss:Convexity2}

We give two more applications of the H\"older continuity here, one pertaining the convexity structure of the regular set and the other to the isometry group of limit spaces.    Recall that it was conjectured in section 4 of \cite{ChC3} that the isometry group for any limit is a Lie group.  In fact, it was proven in section 4 of \cite{ChC3} that the isometry group is a Lie group provided that one could prove that the regular set is connected in a certain weak sense.  It was also shown in \cite{ChC3} that in the non-collapsed case the regular set is connected in this sense and, thus, in the non-collapsed case the isometry group is a Lie group.

Now we introduce two notions of convexity.  We call a $\nu$-measurable set $U$ a.e.-convex if for $\nu\times\nu$ a.e. pair $(x,y)\in U\times U$ it holds that there exists a minimizing geodesic $\gamma\subseteq U$ which connects the two.  For instance consider the example $\RR^n\setminus\{0\}$.

We also consider the notion of a weakly convex subset.  Given two points $x,y\in X$ of a length space $X$ we say that a curve $\gamma$ connecting them is an $\epsilon$-geodesic if
$$||\gamma|-d(x,y)|\leq\epsilon^2 d(x,y)\, .$$
We say that a subset $U\subseteq X$ is weakly convex if the induced length space distance on $U$ is the same as the restricted metric.  In other words if $x$, $y\in U$, then $U$ may not contain a minimizing geodesic connecting $x$ and $y$, but for each $\epsilon>0$ there is an $\epsilon$-geodesic connecting $x$ and $y$ which is contained completely inside $U$.  Again consider the example of $\RR^n\setminus\{0\}$.  Now we state our convexity Theorem.

\begin{theorem}\label{t:weakconvex}
The following hold.
\begin{enumerate}
 \item $\cR_k$ is a.e. convex.
 \item $\cR_k$ is weakly convex.
\end{enumerate}
\end{theorem}

\begin{proof}
The first statement is just a restatement of Lemma \ref{l:collapsedconvex} and Theorem \ref{t:uniquek}. For the second statement note that the a.e-convexity of $\cR_k$ implies that a.e. $z_0\in \cR_k$ has the property that for a.e. $z_1\in\cR_k$ there exists a limit minimizing geodesic $\gamma_{z_0,z_1}$ connecting $z_0$ and $z_1$ such that $\bar\gamma_{z_0,z_1}\subseteq\cR_k$.  Let us denote
\begin{align}
 \cR_k^c\equiv\{z_0:\text{ for a.e. }z_1\in\cR_k \text{ there exists a minimizing }\gamma_{z_0,z_1}\text{ with }\bar\gamma\subseteq\cR_k\}\, ,
\end{align}
as the collection of such $z_0$'s and
\begin{align}
\cR_k^{z_0}\equiv\{z_1:\exists \gamma_{z_0,z_1}\text{ with }\bar\gamma_{z_0,z_1}\subseteq\cR_k\}\, ,
\end{align}
as the corresponding set of $z_1$'s.

Now let $x,y\in\cR_k$ be arbitrary and for $\epsilon>0$ fixed let us define $r_i\equiv \epsilon^2 10^{-i}$.  We will define $\{x_i\},\{y_i\}$ in the following manner.  Let
$$
x_1\in B_{r_1}(x)\cap \cR_k^c\, ,y_1\in B_{r_1}(y)\cap\cR_k^{x_1}\cap \cR_k^c\, ,
$$
with $\gamma_1\subseteq \cR_k$ a unit speed minimizing geodesic connecting them.  Now we define $x_{i+1}$ and $y_{i+1}$ inductively.  Given $x_i$ let
$$
x_{i+1}\in B_{r_{i+1}}(x)\cap\cR_k^{x_i}\cap\cR_k^c\, ,
$$
with
$$\gamma^x_{i+1}\subseteq\cR_k$$
a minimizing geodesic connecting $x_i$ to $x_{i+1}$, and similarly given $y_i$ let
$$
y_{i+1}\in B_{r_{i+1}}(y)\cap\cR_k^{y_i}\cap\cR_k^c\, ,
$$
with
$$
\gamma^y_{i+1}\subseteq\cR_k
$$
a minimizing geodesic connecting them.  Now we can let $\gamma$ be the unit speed curve which is the join of the curves $\{\gamma^x_{i}\}$, $\gamma_1$ and $\{\gamma^y_{i}\}$.  We have that $\gamma$ connects $x$ and $y$ with $\gamma\subseteq\cR_k$ and
\begin{align}
|\gamma|= \sum\left(|\gamma^x_i|+|\gamma^y_i|\right)+|\gamma_1|\leq d(x,y)+\epsilon^2 \, ,
\end{align}
as claimed.
\end{proof}

As a simple consequence of this we have that $\cR_k$ is connected.  This leads us to our next application, which follows from a mildly more uniform version of Lemma \ref{l:collapsedconvex} and Theorem \ref{t:weakconvex}.  In \cite{FY} it was shown that for Alexandrov spaces the isometry group is a Lie group and in \cite{ChC3} it was shown that for non-collapsed limits of manifold with a uniform lower Ricci curvature bound the isometry group is a Lie group.  In fact, in theorem 4.5 in \cite{ChC3} it was shown that for general locally compact metric spaces for which the regular set is dense and where each $\cR_k$ is connected in a weak sense, then the isometry group is a Lie group.  We use a mild generalization of theorem 4.5 in \cite{ChC3} as well as Appendices \ref{s:extendgeodesic} and \ref{s:ReifenbergCollapsed} to prove the following.

\begin{theorem}\label{t:isomliegroup}
The isometry group of a limit space $M_{\infty}$ is a Lie group.
\end{theorem}

\begin{proof}
Let $k$ be the dimension of $M_\infty$, see Theorem \ref{t:uniquek}, with $\cR_k$ the dense collection of $k$-regular points and
\begin{align}
(\cR_k)_{\epsilon,\delta}\subseteq \cR_k \, ,
\end{align}
the subset such that $x\in (\cR_k)_{\epsilon,\delta}$ iff $\forall$ $0<r<\delta$ we have that
$$d_{GH}(B_r(x),B^k_r(0))<\epsilon r\, ,$$
where $B^k_r(0)$ is the $r$-ball in $\RR^k$.  To apply theorem 4.5 of \cite{ChC3} it is enough to find a point $x\in\cR_k$ such that $a.e.$ $y\in \cR_k$ has the property that for every $\epsilon>0$ there exists a $\delta>0$ and a geodesic $\gamma_{x,y}$ connecting $x$ and $y$ such that $\gamma_{x,y}\subseteq(\cR_k)_{\epsilon,\delta}$ (in fact theorem 4.5 requires that such a connectedness property hold for \textit{every} $y\in\cR_k$, however it is easy to check that the proof goes through verbatim if it is only assumed to hold for $a.e.$ $y\in\cR_k$).

To find such an $x\in\cR_k$ we will actually show that $a.e.$  $x\in\cR_k$ has this property.  In fact by Corollary \ref{c:extending} we note that $a.e.$ $ x\in\cR_k$ has the property that for $a.e.$ $y\in\cR_k$ we have a limit minimizing geodesic $\gamma_{x,y}$ such that $x$ and $y$ are interior points of a limit minimizing geodesic $\gamma_{x,y}$.  So fix such an $x\in\cR_k$ and let $\cR_k^x$ be the collection of $y$'s such that there exists such a limit minimizing geodesic $\gamma_{x,y}$.  However, the fact that for every $\epsilon>0$ there exists a $\delta>0$ such that $\gamma_{x,y}\subseteq(\cR_k)_{\epsilon,\delta}$ now follows from Theorem \ref{t:unireif}.
\end{proof}

\subsection{Examples of non--collapsed limit spaces}

Both as a supplement to the main results and to show sharpness of the main results we construct various new examples of non-collapsed limits with lower Ricci bounds in Section \ref{s:ExamplesII}.

Although much more restrictive it is instructive when putting the Theorems and Examples of this paper into context to begin by looking at Alexandrov spaces.  For spaces with such lower sectional bounds it was conjectured by Burago, and proven by Petrunin \cite{Pn}, that tangent cones on the interior of a minimizing geodesic are isometric.  Of course it has been known since \cite{ChC1} for limit spaces with only lower Ricci bounds that tangent cones need not be unique.  In particular there is no hope that tangent cones on the interior of a minimizing geodesic need be unique for limit spaces with only lower Ricci bounds.

However Theorem \ref{t:holderlimit} does potentially suggest a version of Burago's conjecture for limit spaces with only lower Ricci bounds.  Namely, it is reasonable to ask if tangent cones which come from the same sequence of rescalings are unique on the interior of a minimizing geodesic.  Example \ref{ss:Example4} constructs a limit space which shows that this is not the case.  More specifically, Example \ref{ss:Example4} is a non-collapsed limit space $X$ with a minimizing limit geodesic $\gamma\subseteq X$ such that at each distinct point of $\gamma$ the tangent cone is unique but for any $s\neq t$ we have that the tangent cones at $\gamma(s)$ and $\gamma(t)$ are not isometric.  In particular, tangent cones from the same sequence of rescalings along $\gamma$ are not isometric and we see that a generalized Burago conjecture for limits with only lower Ricci bounds cannot hold.

Now although two tangent cones $X_{\gamma(s)}$, $X_{\gamma(t)}$ from the same sequence of rescalings on the interior of a minimizing geodesic $\gamma$ need not be isometric, it does follow from Theorem \ref{t:holderlimit} that they change at a continuous rate, in fact at a $C^{\alpha(n)}$-H\"older rate.  It even follows from the proof that the Gromov-Hausdorff map from $B_1X_{\gamma(s)}$ to $B_1X_{\gamma(t)}$ is $C^{\frac{1}{2}}$ bounded on sets of arbitrarily full measure of $B_1X_{\gamma(s)}$.  The natural question is whether this is a sharp result, and it might even be hoped that the Gromov-Hausdorff maps can be controlled in a Lipschitz fashion.  Example \ref{ss:Example5} shows however that this is not the case, and that in fact the $\frac{1}{2}$-H\"older exponent is sharp.  More specifically, for each $\delta>0$ Example \ref{ss:Example5} constructs a non-collapsed limit space $X_{\delta}$ with a limit minimizing geodesic $\gamma\subseteq X_{\delta}$ so that tangent cones from the same sequence of rescalings along $\gamma$ change at a $C^{1/2}$-H\"older rate but not at a $C^{1/2+\delta}$-H\"older rate.  Thus we will see that Theorem \ref{t:tangentholder} is sharp.  In fact this also has the consequence of showing that the estimates of Section \ref{s:parappxfun} are sharp, namely the estimate $\int_\gamma\fint_{B_r(\gamma(s))}|\Hess_h|^2\leq C$ cannot be replaced with the stronger estimate $\fint_{B_r(\gamma(s))}|\Hess_h|^2\leq C$, where $h$ is the parabolic approximation function.  If it could the techniques of Section \ref{s:GHmap} would show that the Gromov-Hausdorff maps are effectively Lipschitz, which does not hold by the above example.

\subsection{Outline of the Proof of Theorem \ref{t:holder}}

The bulk of the rest of the paper deals with the proof that tangent cones change in a H\"older way along a minimizing geodesic, that is, the proof of Theorem \ref{t:holder}.  

Before getting into the actual strategy of the proof let us try to explain why it might be true and in particular why it requires substantially new estimates.  Consider therefore a minimizing geodesic $\gamma:[0,1]\to M$ which is parametrized by unit speed in an $n$-dimensional manifold $M$ with $\Ric_M\geq -(n-1)$.  By the almost splitting theorem of \cite{ChC1} if $\delta>0$ is fixed and $\delta\leq s<\tau<t\leq 1-\delta$, then for $r_0=r_0(\delta,n)>0$ sufficiently small and $r$ with $r_0>r>0$ the ball $B_{\frac{3r}{2}}(\gamma (\tau))$ almost splits.  This implies that the balls $B_r(\gamma (\tau-\frac{r}{2}))$ and $B_r(\gamma (\tau+\frac{r}{2}))$ are Gromov-Hausdorff close to each other.  This allows one to compare balls with different centers, but with same radii, centered along a minimizing geodesic.  This is exactly what we would like to do.  The downside with this argument is that as the radii become smaller, the distance between the pair of centers of the balls that we can compare become smaller as well.   One may think that if we iterate this argument along the geodesic going from $\gamma (s)$ and to $\gamma (t)$, then perhaps we get the desired estimate.  The issue is that we would have to iterate this process roughly $\frac{t-s}{r}$ many times, while the error induced from each iteration is roughly $r^{\beta}$, where $\beta(n)$ is a small dimensional constant which comes from the Abresch-Gromoll inequality.  Thus, when we iterate this process $\frac{t-s}{r}$ times we get that the scale invariant Gromov-Haudorff distance between the first and last ball is roughly bounded by $\frac{t-s}{r}\,r^{\beta}=(t-s)\,r^{\beta-1}$, which converges to $\infty$ as $r\to 0$ and so in other words blow up.   It is therefore clear that this naive argument does not work, rather we need better estimates that we can integrate up from a neighborhood of $\gamma (s)$ to a neighborhood of $\gamma (t)$.    

To examine and explain the better estimates that we need, we will need to explain how the almost splitting in \cite{ChC1} is proven.  The overall strategy of the proof is to first approximate certain distance functions with functions with better properties\footnote{The basics of this overall strategy was already present in the earlier papers \cite{C1}--\cite{C3}.}.  The better property that was needed in \cite{ChC1} was an $L^2$ bound for the hessian of the approximating functions which distance functions themselves may not have; see remark 4.102 in \cite{ChC1}.  With the hessian bound one can integrate over geodesics and use the outcome in combination with the first variation formula to turn it into information about the Gromov-Hausdorff distance.

Before we explain in more detail the better estimates that we prove here, and how it give us the desired H\"older continuity, let us focus on the crucial $L^2$ bound for the hessian of the approximating functions.   To see what may be possible let us again examine a minimizing geodesic $\gamma :[0,1]\to M$.  Obviously, in the interior of the geodesic the distance to one of the endpoints is a smooth function $d$ and the Bochner formula applied to $d$ is simply the matrix Riccati equation.  Moreover, a simple argument, that we give in Section \ref{s:GHmap} and has in roots in an old paper of E. Calabi, \cite{Ca}, shows that
\begin{align}
\int_{\delta}^{1-\delta}|\Hess_d|^2\leq \frac{C}{\delta}\, .
\end{align}
The key improved hessian estimate that we show is a version of this estimate for the approximating function.    Moreover, it is easy to see that if one applies the Cauchy-Scwarz inequality to this hessian bound and integrates along the geodesic, then one get an infinitesimal version of the desired H\"older estimate.  Where the H\"older exponent $\frac{1}{2}$ comes from is the Cauchy-Schwarz inequality.

We will next be more specific about the proof of Theorem \ref{t:holder} and how it differs from the proofs of the almost maximal, almost metric cone, and almost splitting theorems in \cite{C1}--\cite{C3}, \cite{ChC1}.  

Section \ref{s:HessBounds} below contains the new functional estimates of the paper.  Section \ref{s:Harnack} is dedicated to a new mean value inequality and the proof of a new excess estimate.  Recall that given points $p,q\in M$ we have the excess function
\begin{align}
e(x)\equiv d(p,x)+d(x,q)-d(p,q)\, .
\end{align} 
Along a minimizing geodesic $\gamma$ connecting $p$ and $q$ we have that $e$ takes its minimum value $e|_\gamma\equiv 0$.  A simple estimate using the lipschitz nature of $e(x)$ then gives for $x\in B_r(\gamma(t))$, where $\gamma(t)$ is some interior point of the geodesic, that $e(x)\leq Cr$.  It is an important estimate by Abresch and Gromoll that this can be improved to the statement
\begin{align}
e(x)\leq C\,r^{1+\alpha(n)}\, ,
\end{align}
where $\alpha$ is a small dimensional constant and $x\in B_r(\gamma(t))$.  This estimate plays a key role in the estimates of \cite{ChC1}.  In the proof of Theorem \ref{t:holder} it is important to have an improvement of this estimate.  Namely, if we were on a smooth manifold with {\it bounded curvature}, then we would expect an estimate of the form $e(x)\leq C\,r^2$.  This is because $e(x)$ would be a smooth function which would obtain a minimum at $\gamma(t)$.  In the case of a lower Ricci bound taking $\alpha\equiv 1$ seems to be a strong statement, however Theorem \ref{t:L1_excess} says that if we only ask for this to hold at most points then this is possible.  More precisely it is proved that
\begin{align}
\fint_{B_r(\gamma(t))}e\leq C\,r^2\, .
\end{align}
In fact, combining this with $|\nabla e|\leq 2$ immediately gives a new proof of the Abresch-Gromoll estimate.

Section \ref{s:parappxfun} is then focused on building new approximation functions to the distance function $d_p(x)$. In \cite{ChC1} a key point is to approximate $d_p$ by a harmonic function $h$.  Although $d_p$ is clearly not smooth, they are able to prove useful estimates on this harmonic approximation.  Most importantly, with the help of the Abresch-Gromoll inequality, they prove an estimate on the hessian
\begin{align}
\fint_{B_r(\gamma(t))}|\Hess_h|^2\leq C\,r^{-2(1-\alpha)}\, .
\end{align}
Although this estimate blows up with $r$, it is better than the scale invariant estimate that was proven and used in \cite{C3}.  For our purposes this is not enough, see Section \ref{s:parappxfun} for a more detailed explanation.  It is important in our situation that we be able to take $\alpha\equiv 1$ in order to get the full $L^2$-bound on the hessian; cf. the discussion in the beginning of the subsection.  To make improvements in terms of getting a better hessian bound it will be important to consider a new class of approximating functions.  Instead of the harmonic approximation to $d_p$ we will consider parabolic approximation.  The idea is that if instead of approximating a distance function on a ball with the harmonic function with the same boundary values, then one ought to be doing better if one instead replace the distance function by the function where we flow it by the heat equation.  The harmonic approximation can then be thought of as the limit when one flow to $t\to \infty$.  By flowing for a relative short amount of time the approximation should resemble the original distance more and yet serve as a regularization.   Precisely, we will flow $d_p$ by the heat flow for roughly time $r^2$.  These functions will turn out to have much better properties than the harmonic approximations.  Even so,  the estimate with $\alpha\equiv 1$ may fail to hold for some ball $B_r(\gamma(t))$, see Example \ref{ss:Example5}.  What we will prove in Theorem \ref{t:mainregthm1} is that it holds for {\it most} balls $B_r(\gamma(t))$, that is
\begin{align}
\int_\gamma\fint_{B_r(\gamma(t))}|\Hess_h|^2\leq C\, .
\end{align}
See Theorem \ref{t:mainregthm1} for a more precise statement.

Section \ref{s:GHmap} is then dedicated to finishing the proof of Theorem \ref{t:holder}.  We begin in Section \ref{ss:jacobi} by proving an infinitesimal version of the estimate (mentioned earlier in this subsection), which is itself quite instructive and gives rise to a Jacobi field estimate.  To then extend this infinitesimal version and finish the proof of Theorem \ref{t:holder} we must actually construct mappings between the balls $B_r(\gamma(s))$ and $B_r(\gamma(t))$ and use the estimates of Section \ref{s:HessBounds} to prove the Gromov-Hausdorff properties of these maps.  The map itself will be the gradient flow induced by $-\nabla d_p$.  This mapping is of course only a measurable map, and in general controlling the gradient flow of a Lipschitz function without hessian estimates requires some technical work.  The results of Section \ref{ss:gradientflow} will show that to control this map it is enough to control {\it nearby} smooth functions.  One of the primary technical challenges is to show that {\it most} geodesics which begin and end near one another remain close.  Namely, if $x\in B_r(\gamma(t))$ and $\gamma_{p,x}$ is the geodesic connecting $p$ and $x$, then it is not at all clear that the geodesics $\gamma_{p,x}(u)$ and $\gamma(u)$ even remain near one another as $u$ varies.  In this case the gradient flow mapping map not even map $B_r(\gamma(t))$ near $B_r(\gamma(s))$, much less define a Gromov-Hausdorff approximation.  Section \ref{ss:volcomparison} deals with this issue, which requires essentially every tool developed in this paper.  Finally, Section \ref{ss:proofmaintheorem} finishes the proof of Theorem \ref{t:holder}.

\section{Hessian bounds for approximations of distance functions}\label{s:HessBounds}

Throughout this section $(M^{n},g)$ has $\Ric\geq -(n-1)$ and $p$, $q$ are points in $M$ with $d(p,q)=d_{p,q}\leq 1$.  Obviously, one can consider points further apart by applying the estimates of this section recursively.  We will also assume, for simplicity, that $M$ is complete, though these estimates are purely local and this is much stronger than what is needed.  We will be dealing often with the functions
\begin{itemize}
\item $d^{-}(x)\equiv d(p,x)$,
\item $d^{+}(x)\equiv d(p,q)-d(q,x)$,
\item and the excess function $e(x)\equiv d(p,x)+d(x,q)-d(p,q)\equiv d^-(x)-d^+(x)$.
\end{itemize}

In the proof of the local almost splitting theorem in section 6 of \cite{ChC1} (see also \cite{C3}) a key point is to approximate the partial Busemann functions $d^+$ and $d^-$ by harmonic functions $b^{\pm}$ on a ball $B_{r}$ centered at a point of small excess.  The key observation for those results is that in $B_{r}$ the hessian of the functions $b^{\pm}$ is bounded in a fashion which is better than scale invariant, at least in an $L^2$ sense.  In other words, a scale invariant bound would be that for some constant $C(n)$ one has
$$\fint_{B_{r}}|\Hess_b|^2\leq C\,r^{-2}\, ,$$
for all $0<r<1$\footnote{The scale invariant bound was used in \cite{C3}.}, but in \cite{ChC1} the better estimate
$$\fint_{B_{r}}|\Hess_b|^2\leq C\,r^{-2+\alpha}\, ,$$
is proved for some dimensional $\alpha>0$.  This estimate is enough to prove a local splitting theorem on $B_{r}$.  This key $L^2$ bound on the hessian is then, in those papers, integrated over all geodesic segments within the $B_{r}$ ball, the resulting integral is integrated once more and used in combination with the first variation formula and turned into estimates on distances.

On the other hand one can instead takes two balls far apart relative to $r$.  In particular if one takes balls $B_{r}(x)$ and $B_{r}(y)$ with say $d(x,y)\geq \delta\, d_{p,q}$ and tries to compare these, then the hessian being bounded in a way that is better than scale invariant on each $r$-ball is not enough.  In particular it is not enough to iterate the local splitting theorem.  Using this type of argument and the estimates of \cite{ChC1} the version of Theorem \ref{t:holder} one would obtain is in fact that
$$d_{GH}(B_r(\gamma(s)),B_r (\gamma(t)))\leq C|t-s| \,r^{\frac{\alpha}{2}}\, ,$$
for some small $\alpha>0$.  For any $\alpha<2$ this cannot be used to compare tangent cones.

We need therefore a sharper hessian bound that does not degenerate with $r$.  The sharpest one could hope for in this context is an actual $L^{2}$ bound, namely $\fint_{B_{r}}|\Hess_b|^2\leq C$.  Unfortunately a first difficulty is that Example \ref{ss:Example5} tells us that such a bound can indeed fail under a lower Ricci curvature bound.  What will turn out to be true, and is in fact a sharp estimate in its precise form, is that this sort of bound holds \textit{for most} balls $B_{r}$ along the geodesic.  More precisely,\footnote{Theorem \ref{t:mainregthm1} is actually a bit more general.} if $\sigma$ is a minimizing geodesic between $x$ and $y$, then we have an estimate of the form
$$\int^{(1-\delta) d_{p,q}}_{\delta d_{x,y}}\fint_{B_{r}(\sigma(t))}|\Hess_b|^2\leq C\, .$$
Actually, this estimate only holds for a different comparison function.  In fact, instead of considering the harmonic approximations of $d^\pm$ we will consider parabolic approximations by flowing $d^\pm$ by the heat equation some chosen amount of time\footnote{The chosen time is $r^2$ which is the scale invariant amount of time corresponding to balls of radius $r$.}.  For technical reasons this allows for certain pointwise estimates that may fail for the harmonic approximation.

The improvement that allowed one to go from the scale invariant bound for the hessian of $b$ in \cite{C3} to a better than scale invariant bound in \cite{ChC1}, with a dimensional exponent $\alpha$, came from bringing in the Abresch-Gromoll inequality, \cite{AbGl} and getting a better average gradient bound than that of \cite{C3}.  Recall that the Abresch-Gromoll inequality is a bound for the excess of thin triangles.  The Abresch-Gromoll inequality was used in \cite{ChC1} in combination with the Laplacian comparison theorem to get an improved bound for the average of the difference between the norm of the gradient of the approximation to the distance function and the norm of the gradient of distance function itself (which is of course $1$).  To prove our better hessian bound we begin with a mean value inequality that will allow us to get better bound for the norm of the gradient of the approximation.  It will also give us a better excess bound.  In the second subsection that follows we apply this mean value inequality to get the desired hessian bound.

\subsection{Mean value and integral excess inequalities}         \label{s:Harnack}

In this subsection we will record a direct consequence of the mean value inequality for almost sub-solutions of the heat equation on a Riemannian manifold $(M^{n},g)$ with a lower Ricci curvature bound $\Ric\geq -(n-1)$.  As an application we get an integral inequality for the excess which is sharp.  This excess bound, as well as the mean value inequalities, is used in the next subsection when we prove the estimates for the hessian of the approximations of distance functions and in Section \ref{s:GHmap}.  The next lemma will also be applied in the next subsection to get a good bound for the average of $||\nabla h^{\pm}_t|^2-1|$, where $h^{\pm}_t$ are approximations to the distance functions.  This good bound is one of the keys to get the desired hessian bound.

We will below assume $M$ is complete, however, the estimates are purely local and this is not needed.

\begin{lemma}\label{l:L1harnack_p}
If $u:M\times[0,r^2]\to \RR$ is a non-negative continuous function $u(x,t)=u_t(x)$ with compact support for each fixed time, $0<r<R$, and
$(\partial_t-\Delta)\,u\geq -c_0$ in the distribution sense, then
\begin{equation}
\fint_{B_r(x)}u_0\leq c(n,R)\,\left[u_{r^2} (x)+ c_0\, r^2\right]\, .
\end{equation}
\end{lemma}
\begin{remark}
The above more generally gives that $\fint_{B_r(x)}u_0\leq c(n,R)\,\left[\text{inf}_{y\in B_r(x)}u_{r^2} (y)+ c_0\, r^2\right]$, hence an $L^1$-Harnack inequality.
\end{remark}

The proof relies on the following heat kernel estimate.  The estimate is similar in nature to estimates proved by Li-Yau in \cite{LY}, however the nature of the estimate, which is a little more general than those in \cite{LY}, is such that we are required to use different techniques for the proof.  Namely the existence of a good cutoff functions as in \cite{ChC1} are required.

\begin{lemma}\label{l:kernelestimate}
Let $ H_t(x,y)$ be the heat kernel with $0<r\leq R$ and $t\leq R^2$.  Then we have
\begin{enumerate}
\item if $y\in B_{10\sqrt t}(x)$ then $\frac{c^{-1}(n,R)}{Vol(B_{10\sqrt t}(x))}\leq H_t(x,y)\leq \frac{c(n,R)}{Vol(B_{10\sqrt t}(x))}$.
\item $\int_{M\setminus B_r(x)} H_t(x,y)\,dv_g(y)\leq c(n,R)\,r^{-2}\,t$\, .
\end{enumerate}
\end{lemma}

\begin{proof}
Let $\psi^r$ be a cutoff function on $B_{20r}(x)$ as in \cite{ChC1}, hence $\psi^r(y)=1$ on $B_{10r}(x)$, $\psi^r(y)=0$ outside $B_{20r}(x)$ and we have the estimates $r\,|\nabla\psi^r|$, $r^2\,|\Delta\psi^r|\leq c(n,R)$.  Let
$$
\psi^r_t(y)\equiv \int H_t(y,z)\,\psi^r(z)\, ,
$$
denote the solution to the heat equation, then we have that
\begin{align}
|\Delta \psi^r_t |(y) &= \left| \int \Delta_y  H_t (y,z)\,\psi^r(z) \right|\\
&= \left| \int \Delta_z H_t(y,z)\psi^r(z) \right| =
      \left| \int H_t(y,z)\,\Delta \psi^r(z) \right| \leq cr^{-2}\, ,\notag
\end{align}
where we can interchange the laplacians because $\psi^r$ has compact support.  Thus we get that
\begin{align}
 |\psi^r_t(y)-\psi^r(y)|\leq\int_0^t|\Delta\psi^r_s|ds\leq c\,r^{-2}t\, .
\end{align}
In particular, if we then take $t_r=\frac{1}{2c}r^2$ we get that $|\psi^r_{t_r}(x)-1|\leq \frac{1}{2}$, and hence we get the two equations:
\begin{align}
\frac{1}{2}\leq &\int H_{t_r}(x,z)\,\psi^r(z) \leq \int_{B_{20r}(x)} H_{t_r}(x,z)\, ,\\
&\int_{B_{20r}(x)} H_{t_r}(x,z)\leq \int H_{t_r}(x,z) \leq 1\, .
\end{align}

In particular we find that there must be at least one point $z\in B_{20r}(x)$ such that $\frac{1}{2Vol(B_{10\sqrt t}(x))}\leq H_{t_r}(x,z)\leq \frac{2}{Vol(B_{10\sqrt t}(x))}$.  However, a straight forward application of the Li-Yau Harnack inequality, \cite{LY}, now proves the first statement for $t=r^2$ and any $y\in B_{10r}(x)$.

To prove the second statement the setup is similar.  In this case let $\phi(y)=1-\psi^r(y)$ where $\psi^r$ is now a cutoff function like above with $\psi^r(y)=1$ in $B_{r/2}(x)$ and $\psi^r(y)=0$ outside $B_r(x)$.  If we let $\phi_t$ denote the solution of the heat equation then the same argument as above gives that
$$
\phi_t(x)\leq c(n,R)\,r^{-2}t\, .
$$
Finally this gives us that
\begin{equation}
\int_{M\setminus B_r(x)} H_t(x,y)\,dv_g(y)\leq \int H_t(x,y)\phi(y)dv_g(y) = \phi_t(x)\leq c\,r^{-2}t\, ,
\end{equation}
as claimed.
\end{proof}

Now we can finish Lemma \ref{l:L1harnack_p}

\begin{proof}[Proof of Lemma \ref{l:L1harnack_p}:]
Differentiating, using the heat equation, in particular, that $H$ is a fundamental solution, and integrating by parts yields
\begin{align}
\frac{d}{ds}\left(\int u(y,s)\,H(x,y,r^2-s)\,dy\right)&=\int \partial_tu\,H-\int u\,\partial_tH\\
&=\int \partial_tu\,H-\int u\,\Delta H=\int H\,(\partial_t-\Delta)\,u\geq -c_0\int H=-c_0\, .\notag
\end{align}
Since $u(x,r^2)=\lim_{s\to r^2}\int u(y,s)\,H(x,y,r^2-s)\,dy$ the claim follows by integration provided
\begin{equation}
\int u(y,0)\,H(x,y,r^2)\,dy\geq c\,\fint_{B_r(x)} u(y,0)\,dy
\end{equation}
This however follows by Lemma \ref{l:kernelestimate} since $u\geq 0$ and
\begin{equation}
\inf_{B_r(x)}H(x,\cdot,r^2)\geq \frac{c}{\Vol (B_{r}(x))}\, .
\end{equation}
\end{proof}

Applying Lemma \ref{l:L1harnack_p} to a function that is constant in time gives (cf. theorem 9.22 in \cite{GiTr}):

\begin{corollary}\label{c:L1harnack_e}
If $u\in C_c(M)$ is a non-negative function with $\Delta u \leq c_0$ in the distributional sense, then for each $x\in M$ and $0<r\leq R$
\begin{equation}
\fint_{B_r(x)}u\leq c(n,R)\,\left[u(x)+c_0\, r^2\right]\, .
\end{equation}
\end{corollary}
\begin{remark}
The above more generally gives that $\fint_{B_r(x)}u_0\leq c(n,R)\,\left[\text{inf}_{y\in B_r(x)}u_{r^2} (y)+ c_0\, r^2\right]$, hence an $L^1$-Harnack inequality.
\end{remark}

To use the above to prove the integral excess inequality we need good cutoff functions, which follows from \cite{ChC1} and an standard covering argument.
For a closed subset $C\subseteq M$ and $0<r_0<r_1$ we define the annulus
$A_{r_0,r_1}(C)\equiv T_{r_1}(C)\setminus T_{r_0}(C)$, where $T_{r}(C)$ is the $r$-tubular neighborhood of $C$.

\begin{lemma}\label{l:testfun}
For every $0<r_0< 10\,r_1 \leq R$, there exists a function
$\psi:A_{r_0,r_1}(C)\to \RR$ such that
\begin{enumerate}
\item $\psi\geq 0$ with $\psi(x)=1$ for $x\in A_{3r_0,r_1/3}(C)$
and $\psi(x)=0$ for $x\not\in A_{2r_0,r_1/2}(C)$.
\item $|\nabla \psi |\leq c(n,R)\,r_0^{-1}$ and
$|\Delta\psi|\leq c(n,R)\,r_0^{-2}$ in $A_{2r_0,3r_0}(C)$.
\item $|\nabla \psi |\leq c(n,R)\,r_1^{-1}$ and
$|\Delta\psi|\leq c(n,R)\,r_1^{-2}$ in $A_{r_1/3,r_1/2}(C)$.
\end{enumerate}
\end{lemma}

\begin{proof}
Let $\{x_i\}\in A_{3r_0,r_1/3}(C)$ be some maximal subset of points such that $B_{r_i/16}(x_i)$ are disjoint, where $r_i=d(x_i,C)$.  By maximality $B_{r_i/4}(x_i)$ cover $A_{3r_0,r_1/3}(C)$.  Also if two balls at $x_i$, $x_j$ overlap, then the ratio of $r_i$ and $r_j$ is bounded by $4$, hence the usual volume comparison arguments tell us that the collection $\{B_{r_i/2}(x_i)\}$ overlap at most $c(n,R)$ times at any point.

It follows from theorem 6.33 of \cite{ChC1} that we can construct non-negative functions
$\psi_i:M\to \RR$ with compact support and with $\psi_i = 1$ on $B_{r_i/4}(x_i)$, $\psi_i = 0$ outside $B_{r_i/2}(x_i)$ and $r_i\,|\nabla \psi_i|$, $r_i^2\,|\Delta \psi_i|\leq c(n,R)$.

Consider first the function
$$
\bar\psi(x)=\sum \psi_i(x)\, .
$$
We have that $\bar\psi(x)$ vanishes for $x\not\in A_{2r_0,r_1/2}(C)$ and satisfies the sought after bounds as the supports of each $\psi_i$ intersect each point at most $c(n,R)$ times.  Further, since $\{B_{r_i/2}(x_i)\}$ cover $A_{3r_0,r_1/3}(C)$ we have for $x\in A_{3r_0,r_1/3}(C)$ that $1\leq \bar\psi(x)\leq c(n,R)$, and so if we let $f:[0,\infty)\to \RR$ be a fixed smooth function such that $f(s)=0$ for $s$ near zero and $f(s)=1$ for $s\geq 1$, then $\psi(x)=f(\bar\psi(x))$ is our desired function.
\end{proof}

In the next section we will use Lemma \ref{l:L1harnack_p}
in combination with the following lemma:

\begin{lemma}  \label{l:toapply}
If $h$ solve the heat equation, $\phi\geq 0$ has compact support and is time independent,
$|\phi |$,
$|\nabla \phi |$, $|\Delta \phi|\leq K_1$, and $|\nabla h|\leq K_2$ on $\{\phi>0\}$, then
$(\partial_t-\Delta)\,[|\nabla h|^2\,\phi^2]\leq c=c(n,K_1,K_2)$.
\end{lemma}

\begin{proof}
By the elementary inequality $2ab\leq a^2+b^2$ and since $|\nabla |\nabla h|^2|^2\leq 4\,|\Hess_h|^2\,|\nabla h|^2$ we have
\begin{equation}
4\phi\,|\langle \nabla |\nabla h|^2,\nabla \phi\rangle|\leq 2\epsilon\,\phi^2\,|\nabla |\nabla h|^2|^2
+\frac{2}{\epsilon}\,|\nabla \phi|^2
\leq 8\,\epsilon\,\phi^2\,|\nabla h|^2\,|\Hess_h|^2+\frac{2}{\epsilon}\,|\nabla \phi|^2\, .
\end{equation}
Choose $\epsilon>0$ so small so that $8\,K^2_2\,\epsilon\leq 2$, then by the Botcher formula we have
\begin{align}
\Delta\,[|\nabla h|^2\,\phi^2]&=\phi^2\,\Delta |\nabla h|^2+2\,\langle \nabla |\nabla h|^2,\nabla \phi^2\rangle+|\nabla h|^2\, \Delta \phi^2\notag\\
&\geq 2\phi^2\,|\Hess_h|^2+2\phi^2\,\langle \nabla \Delta h,\nabla h\rangle-2(n-1)\,\phi^2\,|\nabla h|^2+4\phi\,\langle \nabla |\nabla h|^2,\nabla \phi\rangle+|\nabla h|^2\, \Delta \phi^2\\
&\geq 2\phi^2\,\langle \nabla \Delta h,\nabla h\rangle-c\, .\notag
\end{align}
Using that $\Delta h=\partial_t h$, so $\partial_t|\nabla h|^2=2\langle \nabla \Delta h,\nabla h\rangle$, and that $\phi$ is independent of time gives the claim.
\end{proof}

We finish by stating and proving the integral inequality for the excess, which is one of the main results of this Section.  Recall that the excess of $p,q\in M$ is the function
\begin{align}
e_{p,q}(x)\equiv d(p,x)+d(x,q)-d(p,q)\geq 0\, .
\end{align}
If $\gamma(t)$ is a minimizing geodesic connecting $p$ and $q$ then $e$ attains its minimum value $e|_{\gamma}\equiv 0$ on $\gamma$.  If $M$ had uniform estimates on its curvature and injectivity radius, then $e$ would be a smooth function near the interior of $\gamma$, and one would expect for $x\in B_r(\gamma(t))$ the estimate $e(x)\leq Cr^2$.  In the case of only a lower Ricci curvature bound on $M$ this is a lot to ask for, however an important estimate by Abresh and Gromoll gives
\begin{align}
e(x)\leq Cr^{1+\alpha(n)}\, ,
\end{align}
where $\alpha(n)$ is a small dimensional constant and $x\in B_r(\gamma(t))$.  The next Theorem is an improvement of this statement, where we show that even if we can't take $\alpha\equiv 1$, that in fact we can at {\it most} points.  More precisely:

\begin{theorem}\label{t:L1_excess}
Let $p,q\in M$ with $d_{p,q}\equiv d(p,q)\leq 1$ and $0<\epsilon<1$.  If $x\in A_{\epsilon d_{p,q},2d_{p,q}}(\{p,q\})$ satisfies $e(x)\leq r^2\,d_{p,q}\leq \bar r^2(n,\epsilon)\,d_{p,q}$, then
$$
\fint_{B_{r d_{p,q}}(x)} e\leq c(n,\epsilon)\,r^2\,d_{p,q}\, .
$$
\end{theorem}
\begin{remark}
Let us observe that when combined with the estimate $|\nabla e|\leq 2$ that we recover the original Abresch-Gromoll estimate.
\end{remark}
\begin{proof}
Let $\psi$ be given by the previous lemma where $C\equiv\{p,q\}$.  Set $\bar e\equiv \psi e$ and note that
\begin{align}
\Delta \bar e = \Delta\psi\, e + 2\langle\nabla\psi,\nabla e\rangle+\psi \Delta e \leq \frac{c(n,\epsilon)}{d_{p,q}} \, .
\end{align}
This estimate depends on $e$ being appropriately small where $\Delta\psi$ is large and vice versa.  We can therefore apply Corollary \ref{c:L1harnack_e} to $\bar e$ to get the result.
\end{proof}

\subsection{Parabolic approximation}  \label{s:parappxfun}

In this subsection we will show the desired hessian bound $\int^{(1-\delta) d_{p,q}}_{\delta d_{x,y}}\fint_{B_{\epsilon}(\sigma(t))}|\Hess|^2\leq C$ for the parabolic approximation to the distance function.  However, before proving that we will need several lemmas.

Let $\psi^{\pm}:M\rightarrow\RR$ be the cutoff functions given by Lemma \ref{l:testfun} such that for some fixed $\delta>0$ we have
\begin{align}
\psi^{-} = \left\{ \begin{array}{rl}
 1&\mbox{ on }A_{\frac{\delta}{4} d_{p,q},8 d_{p,q}}(p) \notag\\
  0 &\mbox{ on }M\setminus A_{\frac{\delta}{16} d_{p,q},16 d_{p,q}}(p)
       \end{array} \right. ,\,
\psi^{+} = \left\{ \begin{array}{rl}
 1&\mbox{ on }A_{\frac{\delta}{4} d_{p,q},8 d_{p,q}}(q) \notag\\
  0 &\mbox{ on }M\setminus A_{\frac{\delta}{16} d_{p,q},16 d_{p,q}}(q)
       \end{array} \right. \, ,\notag
\end{align}
and let $\psi=\psi^+\,\psi^-$.
Set
\begin{align}
M_{r,s}\equiv A_{ r d_{p,q},s d_{p,q}}(p)\cap A_{r  d_{p,q},s d_{p,q}}(q) \, ,
\end{align}
and let $h^{\pm}_t$ and $e_t$ be solutions to the heat equation $(\partial_t-\Delta) = 0$ on $M$ with  $h^{\pm}_0=\psi\, d^{\pm}$, $e_0=\psi\, e$.  In particular,
$h^{\pm}_0=d^{\pm}$, $e_0=e$ on $M_{\delta/4,8}$, and by uniqueness $e_t=h^-_t-h^+_t$.

We start with the following lemma:

\begin{lemma}                      \label{l:lapcomp}
There exists $c(n,\delta)$ such that
\begin{equation}
\Delta h^-_t\, , \Delta e_t\, , -\Delta h^+_t\leq \frac{c(n,\delta)}{d_{p,q}}\, .
\end{equation}
\end{lemma}

\begin{proof}
We prove it for $e_t$, the proof can be repeated verbatim for the others.  Note first that
\begin{align}\label{e:e_0est}
\Delta e_0 = e\,\Delta\psi+2\langle\nabla\,\psi,\nabla e\rangle+\psi\,\Delta e \leq \frac{c(n,\delta)}{d_{p,q}}\, .
\end{align}
This inequality holds distributionally, and because $e_0$ has compact support it holds distributionally on functions which themselves may not have compact support.  Now if $H_t(x,y)=H(x,y,t)$ is the heat kernel, then $e_t(x)= \int_{M_{\delta/16,16}} H_t(x,y)\,\psi(y)\,e(y)\,dv_g(y)$.  Thus we get
\begin{align}
\Delta e_t(x) &= \int_{M_{\delta/16,16}}\Delta_x\, H_t(x,y)\,\psi(y)\,e(y)\,dy
 = \int_{M_{\delta/16,16}}\Delta_y\, H_t(x,y)\,\psi(y)\,e(y)\,dy\leq \frac{c(n,\delta)}{d_{p,q}}\, .
\end{align}
\end{proof}

From this we get:

\begin{lemma}                       \label{l:Eest}
There exists $c(n,\delta)$ such that for each $\epsilon\leq\bar\epsilon(n,\delta)$ and $x\in M_{\delta/2,4}$ the following holds for each $y\in B_{10d_{\epsilon}}(x)$, where $d_{\epsilon}=\epsilon\, d_{p,q}$:
\begin{enumerate}
\item $|e_{d_{\epsilon}^2}(y)|\leq c\left(\epsilon^2 d_{p,q}+e(x)\right)$.
\item $|\nabla e_{d_{\epsilon}^2}|(y)\leq c\left(\epsilon +\frac{\epsilon^{-1}e(x)}{d_{p,q}}\right)$.
\item $|\frac{d}{dt}e_{d_{\epsilon}^2}|(y)$,  $|\Delta e_{d_{\epsilon}^2}|(y)\leq c\left(\frac{1}{d_{p,q}}+\frac{\epsilon^{-2}e(x)}{d^2_{p,q}}\right)$.
\item $\fint_{B_{d_{\epsilon}}(y)}|\Hess_{e_{d_{\epsilon}^2}}|^2 \leq c\left(\frac{1}{d_{p,q}}+\frac{\epsilon^{-2}e(x)}{d_{p,q}^2}\right)^2$.
\end{enumerate}
\end{lemma}

\begin{remark}
In particular, if we have the pointwise estimate $e(x)\leq \epsilon^2 d_{p,q}$, then on $B_{\epsilon d_{p,q}}(x)$ we have the inequalities
$|e_{d_{\epsilon}^2}|\leq c\epsilon^2 d_{p,q}$, $|\nabla e_{d_{\epsilon}^2}|\leq c\epsilon$, and $|\frac{d}{dt}e_{d_{\epsilon}^2}|$, $|\Delta e_{d_{\epsilon}^2}|$, $\fint_{B_{d_{\epsilon}}}|\Hess_{e_{d_{\epsilon}^2}}|^2 \leq \frac{c}{d_{p,q}}$.
\end{remark}

\begin{proof}
It follows from the previous lemma that
\begin{align}
\Delta e_t(x)\leq \frac{c}{d_{p,q}} \, .
\end{align}
By the definition of the heat equation this means that
\begin{align}
 e_t(x)=e_0(x)+\int_0^t(\Delta e_s)\,ds\leq e(x)+\frac{c}{d_{p,q}} t\, .
\end{align}
Hence, for all $t\in [d_{\epsilon}^2/4,4d_{\epsilon}^2]$ we get that
\begin{align}
e_t(x)\leq e(x)+c\epsilon^2 d_{p,q}\, .
\end{align}
In particular, we can apply the Li-Yau Harnack inequality, \cite{LY}, for $t=d_{\epsilon}^2$ and $y\in B_{10d_{\epsilon}}(x)=B_{10\sqrt t}(x)$ to get that there is a constant $c(n,\delta)$ such that $e_{d_{\epsilon}^2}(y)\leq c\left(\epsilon^2 d_{p,q}+e(x)\right)$, which proves the first statement.  To see the third statement observe first that the two claims it consists of are equivalent since $e_t$ satisfies the heat equation.  Using this the third statement follows from the Li-Yau gradient estimate combined with the previous lemma and the first statement.  The second statement follows from the first statement together with the equivalence of the inequalities in the third statement,  the previous lemma, and the Li-Yau gradient estimate, \cite{LY}.  The final statement is proved using a Bochner formula as in \cite{ChC1} (see pages 217--218 and 228 there),  \cite{C1}--\cite{C3}, and also Theorem \ref{t:mainregthm1}.
\end{proof}

We can now begin to estimate the approximation functions $h^{\pm}$ themselves:

\begin{lemma}\label{l:he_est}
There exists $c(n,\delta)$ such that for every $\epsilon\leq\bar\epsilon(n,\delta)$ and $x\in M_{\delta/2,4}$ we have
\begin{equation}
|h^{\pm}_{d_{\epsilon}^2}-d^{\pm}|(x)\leq c\left(\epsilon^2\, d_{p,q}+e(x)\right)\, .
\end{equation}
\end{lemma}

\begin{remark}
A first important difference between the parabolic approximations $h^{\pm}$ and the harmonic approximations of \cite{ChC1}, \cite{C3} is that the error term above is controlled \textit{pointwise} by $e$, as opposed to the $L^{\infty}$ norm of $e$ on $B_{\epsilon d_{p,q}}(x)$.
\end{remark}

\begin{proof}
As in the proof of the last lemma we see because
\begin{align}
 \Delta h^{-}\leq \frac{c}{d_{p,q}}\, ,
\end{align}
and
\begin{align}
 -\frac{c}{d_{p,q}}\leq \Delta h^+\, ,
\end{align}
that from Lemma \ref{l:lapcomp} we immediately get for every $x\in M_{\delta/2,4}$ that
\begin{align}
 h^{-}_{d_{\epsilon}^2}(x)\leq d^-(x)+c\,\epsilon^2\, d_{p,q}\, ,
\end{align}
and
\begin{align}
d^+(x)-c\,\epsilon^2\, d_{p,q}\leq h^+_{d_{\epsilon}^2}(x) \, .
\end{align}
These are equivalent to $h^{-}_{d_{\epsilon}^2}(x)- d^-(x)\leq c\,\epsilon^2\, d_{p,q}$ and $-(h^+_{d_{\epsilon}^2}(x)-d^+(x))\leq c\,\epsilon^2\, d_{p,q}$.  The reverse inequalities follow from
\begin{align}
h^{-}_{d_{\epsilon}^2}(x)- d^-(x)=h^{+}_{d_{\epsilon}^2}(x)- d^+(x)+e_{d^2_{\epsilon}}(x)-e(x) \, ,
\end{align}
since by the last lemma
\begin{align}
 |e_{d_{\epsilon}^2}|\leq c\left(\epsilon^2\, d_{p,q}+e(x)\right)\, .
\end{align}
\end{proof}

An obvious, but important, corollary of the last lemma is that $h^{\pm}_{d_{\epsilon}^2}$ and $d^{\pm}$ are automatically close at $x$ when the excess $e(x)$ is small.  More generally, we would like to prove smallness results along $\epsilon$-geodesics.  Recall that an $\epsilon$-geodesic between $p$ and $q$ is simply a unit speed curve $\sigma$ such that $||\sigma|-d(p,q)|\leq\epsilon^2 d_{p,q}$.  The following obvious lemma tells us the basics of what we need to know about $\epsilon$-geodesics:

\begin{lemma}                       \label{l:eps_geo}
The following statements hold:
\begin{enumerate}
\item Let $\sigma$ be an $\epsilon$-geodesic connecting $p$ and $q$ and let $z\in\sigma$, then $e(z)\leq\epsilon^2\, d_{p,q}$.
\item Let $x\in M$ such that $e(x)\leq \epsilon^2\, d_{p,q}$, then there exists an
$\epsilon$-geodesic $\sigma$ such that $x\in\sigma$.
\end{enumerate}
\end{lemma}

We can now prove the promised corollary:

\begin{corollary}                       \label{c:|h-d|gamma}
There exists $c(n,\delta)$ so that for every $\epsilon$-geodesic $\sigma$ connecting $p$, $q$, and any $x\in \sigma\cap M_{\delta/2,4}$
\begin{equation}
|h^{\pm}_{d_{\epsilon}^2}-d^{\pm}|\leq c\,\epsilon^2\, d_{p,q}\, .
\end{equation}
This can equivalently be stated that for each $t$ with $\frac{\delta}{2}d_{p,q}<t<(1-\frac{\delta}{2})\,d(p,q)$ we have that
\begin{equation}
|h^{\pm}_{d_{\epsilon}^2}(\sigma(t))-t|\leq c\,\epsilon^2 d_{p,q}\, .
\end{equation}
\end{corollary}

\begin{proof}
By Lemma \ref{l:he_est}, and because $\sigma$ is unit speed, the statements hold for each $z$ with $e(z)\leq \epsilon^2 d_{p,q}$ and $z\in M_{\delta/2,4}$.  However, by Lemma \ref{l:eps_geo},  the excess estimate does, in fact, hold for each $z\in\sigma$ as claimed.
\end{proof}

The next lemma gives a sharp upper bound for the gradient of $h^{\pm}$:

\begin{lemma}                              \label{l:sharp|gradh|}
There exists $c(n,\delta)$ such that for all $x\in M_{\delta/2,4}$ and $\epsilon\leq\bar\epsilon(n,\delta)$ we have that
\begin{equation}  \label{e:sharp|gradh|}
|\nabla h_{d_{\epsilon}^2}^\pm|\leq 1+c\,d_{\epsilon}^2\, .
\end{equation}
\end{lemma}

\begin{proof}
Note first that
\begin{align}
&|\nabla h^{\pm}_0| \leq |\nabla\psi\,|\,d^{\pm}+\psi\,|\nabla d^{\pm}|  \leq c(n)\, ,\notag\\
&|\nabla h^\pm_0|=1 \, \text{ on }\,  M_{\delta/4,8}\, ,\\
&|\nabla h^\pm_0|=0 \, \text{ outside }\, M_{\delta/16,16}\notag\, .
\end{align}
By the Bochner formula we can choose $c(n)$ so that
\begin{align}
(\partial_t-\Delta)[\text{e}^{-ct}\,|\nabla h^\pm_t|^2]\leq 0 \, .
\end{align}
Thus by the parabolic maximum principle and the Li-Yau, \cite{LY}, upper bound for the heat kernel for all $x\in M_{\delta/2,4}$ and $t\leq 4\,d_{\epsilon}^2$
\begin{align}
\text{e}^{-ct}|\nabla h^\pm_t|^2(x)\leq \int_{M_{\delta/16,16}} H_{t}(x,y)\,|\nabla h^\pm_0|^2(y)
\leq \int_{M_{\delta/4,8}} H_{t}(x,y) + c\int_{M_{\delta/16,16}\setminus M_{\delta/4,8}} H_{t}(x,y) \leq 1+c\,t\, .
\end{align}
This implies (\ref{e:sharp|gradh|}) as claimed.
\end{proof}

We will next combine the above with Corollary \ref{c:|h-d|gamma} and Lemma \ref{l:L1harnack_p} to get:

\begin{theorem} \label{t:L1|gradh|}
There exists a constant $c(n,\delta)$ such that for all $\epsilon\leq\bar\epsilon(n,\delta)$ we have
\begin{enumerate}
\item If $x\in M_{\delta,2}$ with $e(x)\leq\epsilon^2\, d_{p,q}$, then
$\fint_{B_{10d_{\epsilon}}(x)}||\nabla h^\pm_{d_{\epsilon}^2}|^2-1|\leq c\,\epsilon $.\label{e:item1}
\item If $\sigma$ is an $\epsilon$-geodesic connecting $p$ and $q$,
then $\fint_{\delta d_{p,q}}^{(1-\delta)d_{p,q}}\fint_{B_{10d_{\epsilon}}(\sigma(s))}||\nabla h^\pm_{d_{\epsilon}^2}|^2-1|\leq c\,\epsilon^2 $.\label{e:item2}
\end{enumerate}
\end{theorem}

\begin{proof}
We will prove the claims for $h^-$, the argument for $h^+$ is the same with obvious changes.  Set
$$
w_t=1+ct-|\nabla h^-_t|^2\, ,
$$
where $c(n,\delta)$ is chosen from the last lemma so that $w_t\geq 0$ on $M_{\delta/2,4}$.  It follows from Lemma \ref{l:toapply} that
\begin{align}
(\partial_t-\Delta)\,[\phi^2\,w_t]\geq -c \, ,
\end{align}
where $\phi=\phi^+\,\phi^-$ and $\phi^{\pm}$ are given by Lemma \ref{l:testfun} similarly to $\psi^{\pm}$ except $\phi=1$ on $M_{\delta,2}$ and $\phi=0$ outside $M_{\delta/2,4}$.  By Lemma \ref{l:L1harnack_p} for all $y\in M_{\delta/2,4}$
\begin{align}
\fint_{B_{10\sqrt t}(y)}w_t\leq c\,\left(\inf_{B_{10\sqrt t}(y)} w_{2t}+t\right) \, .
\end{align}
Set $t=d_{\epsilon}^2$.  To complete the proof we need to show there is a point in $B_{10d_{\epsilon}}(y)$ where $w_{2t}$ is small.  To do this let $\sigma$ be an $\epsilon$-geodesic connecting $p$ and $q$.  In \ref{e:item1}) assume $\sigma$ is the piecewise geodesic passing through $x$ as in Lemma \ref{l:eps_geo}.  To prove \ref{e:item1}) note that by Corollary \ref{c:|h-d|gamma}
\begin{align}
&|h^-_{2d_{\epsilon}^2}(x)-h^-_{2d_{\epsilon}^2}(\sigma(d_{p,x}-10\,d_{\epsilon}))-10\,d_{\epsilon}|\, ,\notag\\
&\leq |d^-(x)-d^-(\sigma(d_{p,x}-10\,d_{\epsilon}))-10\,d_{\epsilon}|+c\,\epsilon^2\, d_{p,q} = c\,\epsilon^2\, d_{p,q}\, .
\end{align}
Combining this with the Cauchy-Schwarz inequality and the fundamental theorem of calculus gives
\begin{align}
\int_{d_{p,x}-10d_{\epsilon}}^{d_{p,x}}w_{2d_{\epsilon}^2} &=\int_{d_{p,x}-10d_{\epsilon}}^{d_{p,x}}\left(1+c\,d_{\epsilon}^2-|\nabla h^-_{2d_{\epsilon}^2}|^2\right)\,(\sigma(s))ds\leq 10\,d_{\epsilon} + cd_{\epsilon}^3 - \frac{1}{10d_{\epsilon}}\,\left(\int_{d_{p,x}-10\,d_{\epsilon}}^{d_{p,x}}\nabla_{\dot\sigma}\,h^-_{2d_{\epsilon}^2} ds\right)^2\notag\\
&=10\,d_{\epsilon} +cd_{\epsilon}^3 - \frac{1}{10d_{\epsilon}}\,\left(h^-_{2d_{\epsilon}^2}(\sigma(d_{p,x}))-h^-_{2d_{\epsilon}^2}\,(\sigma(d_{p,x}-10\,d_{\epsilon}))\right)^2\\
&\leq 10\,d_{\epsilon} -10\,d_{\epsilon} + 2c\,\epsilon^2\, d_{p,q}+c^2\epsilon^4\,d_{p,q}^2 =
2c\,\epsilon^2\, d_{p,q}+c^2\epsilon^4\,d_{p,q}^2\, .\notag
\end{align}
In particular, there is some point of $\sigma (d_{p,x}-10d_{\epsilon},d_{p,x})$ with  $w_{2d_{\epsilon}^2}\leq c\epsilon$. From this the first statement follows.

The argument for  \ref{e:item2}) is similar.  As before, by Corollary \ref{c:|h-d|gamma} for all
$s$ with $\delta\, d_{p,q}<s<(1-\delta)\,d_{p,q}$
\begin{align}
\left|\left(h^-_{2d_{\epsilon}^2}(\sigma(s))-h^-_{2d_{\epsilon}^2}(\sigma(\delta d_{p,q}))\right)-\left(s-\delta d_{p,q}\right)\right|\leq c\,\epsilon^2\, d_{p,q}\, .
\end{align}
Arguing as before we get that
\begin{equation}
\int_{\delta d_{p,q}}^{(1-\delta)d_{p,q}}w_{2d_{\epsilon}^2}(\sigma(s))\,ds
\leq c\,\epsilon^2\, d_{p,q}\, ,
\end{equation}
and
\begin{align}
&\int^{(1-\delta)d_{p,q}}_{\delta d_{p,q}}\fint_{B_{10d_{\epsilon}}(\sigma(s))}||\nabla h_{d_{\epsilon}^2}^-|^2-1|\leq \int^{(1-\delta)d_{p,q}}_{\delta d_{p,q}}\fint_{B_{10d_{\epsilon}}(\sigma(s))}w_{d_{\epsilon}^2} +c\,\epsilon^2d_{p,q}\\
&\leq c\int^{(1-\delta)d_{p,q}}_{\delta d_{p,q}}
\inf _{B_{10d_{\epsilon}}(\sigma(s))} w_{2d_{\epsilon}^2}
+c\,\epsilon^2d_{p,q}\leq c\int^{(1-\delta)d_{p,q}}_{\delta d_{p,q}}w_{d_{\epsilon}^2}(\sigma(s))+c\,\epsilon^2d_{p,q}\leq c\,\epsilon^2d_{p,q}\, .\notag
\end{align}
\end{proof}

We will next use the above estimates to prove the following main estimate, for convenience we repeat a few previous estimates to have them collected under one theorem (given the estimates above the proof follows very similar arguments in \cite{C1}--\cite{C3}, \cite{ChC1}):

\begin{theorem}                          \label{t:mainregthm1}
There exists a constant $c(n,\delta)$ such that for all $\epsilon\leq\bar\epsilon(n,\delta)$, any $x\in M_{\frac{\delta}{2},4}$ with $e(x)\leq \epsilon^2\, d_{p,q}$ and any $\epsilon$-geodesic $\sigma$ connecting $p$ and $q$ there exists $r\in [\frac{1}{2},2]$ with
\begin{enumerate}
\item $|h^\pm_{rd_{\epsilon}^2}-d^\pm|(x)\leq c\,\epsilon^2\, d_{p,q}$.
\label{e:emain1}
\item $\fint_{B_{d_{\epsilon}}(x)}||\nabla h^\pm_{rd_{\epsilon}^2}|^2-1| \leq c\,\epsilon$.\label{e:emain2}
\item $\int^{(1-\delta)d_{p,q}}_{\delta d_{p,q}}\fint_{B_{d_{\epsilon}}(x)}||\nabla h^\pm_{rd_{\epsilon}^2}|^2-1| \leq c\,\epsilon^2$.
\label{e:emain3}
\item $\int^{(1-\delta)d_{p,q}}_{\delta d_{p,q}}\fint_{B_{d_{\epsilon}}(\sigma(s))}|\Hess_{h^\pm_{rd_{\epsilon}^2}}|^2\leq \frac{c}{d^2_{p,q}}$.
\label{e:emain4}
\end{enumerate}
\end{theorem}

\begin{proof}
\ref{e:emain1}) is Corollary \ref{c:|h-d|gamma} while \ref{e:emain2}), \ref{e:emain3}) are contained in Theorem \ref{t:L1|gradh|}. The proof of \ref{e:emain4}) uses the Bochner formula as in \cite{ChC1} (see also \cite{C1}--\cite{C3} for a very similar argument).  To begin with for any $\sigma(s)$ it follows from theorem 6.33 of \cite{ChC1} that we can construct a cutoff function $\phi$ such that $\phi(y)=1$ on $B_{d_{\epsilon}}(\sigma(s))$ and vanishes outside $B_{3d_{\epsilon}}(\sigma(s))$ while satisfying the estimates $d_{\epsilon}\,|\nabla\phi|$, $d_{\epsilon}^2\,|\Delta \phi |\leq c(n)$.  Further let $a(t)$ be a smooth function in time with $0\leq a\leq 1$ and $a(t)=1$ for $t\in [\frac{1}{2}d_{\epsilon}^2,2d_{\epsilon}^2]$, vanishing for $t\not\in[\frac{1}{4}d_{\epsilon}^2,4d_{\epsilon}^2]$ and satisfying $|a'|\leq 10d_{\epsilon}^{-2}$.  By the Bochner formula and since $(\partial_t-\Delta)\,h^{\pm}=0$ we have
\begin{equation}
-\frac{1}{2}(\partial_t-\Delta)\,\left(|\nabla h^{\pm}|^2-1\right)=
-\frac{1}{2}(\partial_t-\Delta)\,|\nabla h^{\pm}|^2=|\Hess_{h^{\pm}}|^2
+\Ric (\nabla h^{\pm},\nabla h^{\pm})\, .
\end{equation}
Multiplying by $2\,a(t)\,\phi(y)$ and integrating
we see that for each $t$
\begin{align}
&2\int_{M} a(t)\,\phi\,|\Hess_{h_t^{\pm}}|^2
= \int_{M} a(t)\,\phi\,\Delta \left( |\nabla h_t^{\pm}|^2-1\right)
- 2\int_{M} a(t)\,\phi\,\Ric(\nabla h_t^{\pm},\nabla h_t^{\pm}) -\int_{M} a(t)\,\phi\,\partial_t \left(|\nabla h_t^{\pm}|^2-1\right)\notag\\
&= \int_{M} a(t)\, \left(|\nabla h_t^{\pm}|^2-1\right)\,\Delta \phi
- 2\int_{M} a(t)\,\phi\,\Ric(\nabla h_t^{\pm},\nabla h_t^{\pm}) -\int_{M} a(t)\,\phi\,\partial_t \left(|\nabla h_t^{\pm}|^2-1\right)
\, .
\end{align}
For the last equality we integrated by parts (in space).   It follows that
\begin{align}
2\int_{B_{d_{\epsilon}}(\sigma(s))} &a(t)\,|\Hess_{h_t^{\pm}}|^2 \\
&\leq \frac{c}{d_{\epsilon}^2}\int_{B_{3d_{\epsilon}}(\sigma(s))}
\left| |\nabla h_t^{\pm}|^2-1\right|
+2(n-1)\int_{B_{3d_{\epsilon}}(\sigma(s))} |\nabla h_t^{\pm}|^2
-\int_{B_{3d_{\epsilon}}(\sigma(s))} a(t)\,\phi\,\partial_t \left(|\nabla h_t^{\pm}|^2-1\right)
\, .\notag
\end{align}
Integrating over time, integrating by parts (in time), and using the Bishop-Gromov volume comparison theorem to bound the volume of the ball $B_{3d_{\epsilon}}(\sigma(s))$ by the volume of the concentric ball $B_{d_{\epsilon}}(\sigma(s))$ yields
\begin{align}
\int_{\frac{1}{2}d_{\epsilon}^2}^{2d_{\epsilon}^2}\left(\fint_{B_{d_{\epsilon}}(\sigma(s))} |\Hess_{h_t^{\pm}}|^2\right) \,dt\leq c\,d_{\epsilon}^{-2}\int_{\frac{1}{4}d_{\epsilon}^2}^{4d_{\epsilon}^2}\left(\fint_{B_{3d_{\epsilon}}(\sigma(s))} ||\nabla h_t^{\pm}|^2-1|+c\,d_{\epsilon}^2\right)\,dt\, .
\end{align}
Now this inequality holds for each $s\in[\delta d_{p,q},(1-\delta)d_{p,q}]$ and hence if we integrate over this interval we get
\begin{align}
\int_{\frac{1}{2}d_{\epsilon}^2}^{2d_{\epsilon}^2}\left(\int^{(1-\delta)d_{p,q}}_{\delta d_{p,q}}\fint_{B_{d_{\epsilon}}(\sigma(s))} |\Hess_{h_t^{\pm}}|^2\right)\,dt
\leq cd_{\epsilon}^{-2}\int_{\frac{1}{4}d_{\epsilon}^2}^{4d_{\epsilon}^2}\left(\int^{(1-\delta)d_{p,q}}_{\delta d_{p,q}}\fint_{B_{3d_{\epsilon}}(\sigma(s))} ||\nabla h_t^{\pm}|^2-1|+cd_{\epsilon}^2\right)\,dt\, .
\end{align}
Hence, for some $r\in[\frac{1}{2},2]$ the claim holds for $t=r\,d_{\epsilon}^2$.
\end{proof}

We conclude this section with some estimates, which will be useful in Section \ref{s:GHmap}.

\begin{lemma}\label{l:geo_gradh}
Let $x\in M_{\delta,2}$ with $\sigma_x$ a unit speed minimizing geodesic from $p$ to $x$.  Then for any $\delta\leq s<t\leq d_{p,x}$ the following estimates hold:
\begin{enumerate}
 \item $\int_\delta^{d_{p,x}}||\nabla h_{r^2}^{-}|^2-1|\leq \frac{c(n,\delta)}{d_{p,q}}\,(e(x)+r^2)$.
 \label{e:egeo1}
 \item $\int_\delta^{d_{p,x}}|\langle\nabla h_{r^2}^{-},\nabla d^{-}\rangle-1|\leq \frac{c(n,\delta)}{d_{p,q}}\,(e(x)+r^2)$.
 \label{e:egeo2}
 \item $\int_{s}^{t}|\nabla h_{r^2}^{-}-\nabla d^{-}|\leq \frac{c(n,\delta)\sqrt{t-s}}{\sqrt{d_{p,q}}}\,
 (\sqrt{e(x)}+r)$.
 \label{e:egeo3}
\end{enumerate}
\end{lemma}

\begin{proof}
\ref{e:egeo1}) and \ref{e:egeo2}) are contained in the proof of Theorem \ref{t:L1|gradh|} above.  For \ref{e:egeo3}) note that
\begin{align}
|\nabla h^{-}- \nabla d^{-}|^2 = |\nabla h^{-}|^2+1-2\langle\nabla h^{-},\nabla d^{-}\rangle\leq ||\nabla h^-|^2-1|+2\,|\langle \nabla h^-,\nabla d^-\rangle -1|\, .
\end{align}
Combining this with \ref{e:egeo1}), \ref{e:egeo2}), and the Cauchy-Schwarz inequality gives \ref{e:egeo3}).
\end{proof}

\section{Gromov--Hausdorff approximations}           \label{s:GHmap}

This section is dedicated to completing the proof of Theorem \ref{t:holder}.  Throughout this section $(M^n,g)$ satisfies $\Ric\geq -(n-1)$ and $\gamma:[0,1]\rightarrow M$ is a unit speed minimizing geodesic with $\gamma(0)=p$ and $\gamma(1)=q$.  For points $\gamma(s)$, $\gamma(t)\in \gamma([\delta,1-\delta])$, in order to prove Theorem \ref{t:holder} we will need to construct a Gromov-Hausdorff map between the balls $B_r(\gamma(s))$ and $B_r(\gamma(t))$.  To construct this map we will flow by the gradient of the distance function $-\nabla d_p$.  Of course, this gradient flow is not well defined at every point, and the distance function is far from a smooth function, both of which cause certain technical difficulties.  These difficulties will be addressed in Section \ref{ss:gradientflow}.  Even if these basic difficulties were to be ignored, the most troublesome issue is that if $z\in B_{r}(\gamma(t))$ and $\gamma_{p,z}$ is a minimizing geodesic connecting $p$ and $z$, then there is no reason at all $\gamma_{p,z}(u)$ needs to remain near $\gamma(u)$ for $u$ not near $t$.  In this case the gradient flow map doesn't even map $B_r(\gamma(t))$ near $B_r(\gamma(s))$, much less construct for us a Gromov-Hausdorff map.  We will show in Section \ref{ss:volcomparison} that for a set of large measure in $B_r(\gamma(t))$ that the mentioned geodesics $\gamma_{p,z}$ in fact will remain near $\gamma$.  Finally in Section \ref{ss:proofmaintheorem} we will finish the proof of Theorem \ref{t:holder}.

Let us begin with a couple definitions that will be used throughout this section.  First let
\begin{align}
 \psi_s:M\rightarrow M\, ,
\end{align}
be the gradient flow defined by $-\nabla d_p$.  It is understood that $\psi_s$ is a measurable map which is defined only on a set of full measure.  A main technical issue to be dealt with is knowing that \textit{most} points near $\gamma$ remain near $\gamma$ under this flow.  For this reason we are interested in the following sets:

\begin{definition}
 For $0<s<t<1$ define the set $\mathcal{A}^t_s(r)\equiv \{z\in B_r(\gamma(t)):\psi_u(z)\in B_{2r}(\gamma(t-u))\,\text{  }\, \forall 0\leq u\leq s\}$.
\end{definition}

So $\mathcal{A}^t_s(r)$ defines the set of points in $B_r(\gamma(t))$ which remain a distance of $2r$ from $\gamma$ through the gradient flow, at least up until time $s$.  We will show the volume of $\mathcal{A}^t_s(r)$ is a H\"older function of $s$.

\subsection{Hessian bound along a geodesic and consequences}  \label{ss:jacobi}

In this short subsection we give an $L^2$ bound for the hessian of the distance function to the end point $p=\gamma (0)$ of a minimizing geodesic segment $\gamma:[0,1]\to M$.  This $L^2$ bound holds in a manifold with $\Ric\geq -(n-1)$ and it is a infinitesimal version of the $L^2$ that we obtained in the previous section (it should be compared with theorem 2 of \cite{Ca} and, in particular, its proof, see, for instance, page 674 there).  As a direct consequence of this we get a bound for the distortion of distances along the geodesic that is the infinitesimal version of the desired H\"older bound.  The problem of course is that the bound is infinitesimal and sufficiently small here may depend on the manifold and geodesic in question, which is not terribly useful.  The estimates of Theorem \ref{t:mainregthm1} may be viewed as a non-local version of this, and in a sense the entire purpose of these estimates and the constructions of this section are about taking the following basic infinitesimal estimate and making it less local in nature.  Nonetheless, this estimate will be directly used in the proof of Proposition \ref{p:volratio1}.

\begin{lemma}\label{l:hessianbound}
Let $\gamma:[0,1]\to M$ be a minimizing geodesic as above, $p=\gamma (0)$ and $q=\gamma (1)$, then
\begin{align}
 \int_\delta^{1-\delta}|\Hess_{d_p}|^2\leq \frac{c(n)}{\delta}\, .
\end{align}
\end{lemma}

\begin{proof}
If $d_p(x)$ is the distance function to $p$ then on $\gamma([\delta,1-\delta])$ we have the estimate
$$
|\Delta d_p|\leq \frac{c(n)}{\delta}\, .
$$
The upper bound is the usual comparison principle while the lower bound follows because $d_p(x)+d_q(x)$ obtains a smooth minimum on $\gamma$, hence
$$\Delta d_p \geq -\Delta d_q \geq -\frac{c(n)}{\delta}$$
on $\gamma([\delta,1-\delta])$ as claimed.  Thus we can integrate the equation
$$
\frac{d}{dt}\Delta d_p(\gamma(t)) + |\Hess_{d_p}|^2(\gamma(t))\leq n-1
$$
to get the claim.
\end{proof}

Integrating this lemma and using the Cauchy-Schwarz inequality gives:

\begin{corollary}\label{l:jacobiest}
If $J$ is a Jacobi field on $\gamma$ which vanishes at $p$ and $s,t\in [\delta,1-\delta]$, then
\begin{align}
1-\frac{c(n)}{\sqrt{\delta}}\sqrt{t-s}\leq \frac{|J|(t)}{|J|(s)}\leq 1+\frac{c(n)}{\sqrt{\delta}}\sqrt{t-s}\, .
\end{align}
\end{corollary}

\begin{proof}
Since $\frac{d}{dt}|J|^2 = \Hess_{d_p}(J,J)$ we get from the lemma that 
\begin{align}
\left|\frac{d}{dt}\log |J|^2\right| \leq |\Hess_{d_p}|\, ,
\end{align}
which implies that
\begin{align}
\left|\log \frac{|J|^2(t)}{|J|^2(s)}\right| \leq \int_s^t|\Hess_{d_p}|\leq \sqrt{t-s}\sqrt{\int_\delta^{1-\delta}|\Hess_{d_p}|^2}\leq \frac{c(n)}{\sqrt{\delta}}\sqrt{t-s}\, .
\end{align}
From this the claim easily follows.
\end{proof}

\subsection{The gradient flow}\label{ss:gradientflow}

This subsection is dedicated to addressing the issue that we are flowing by the gradient flow of a function which is not smooth.  We begin with the next lemma, which in essence tells us that we do not need good estimates on $d_p$ in order to take its gradient flow.  Instead, we only need to know that there exists nearby functions for which we have the required estimates.  A related estimate was shown in \cite{ChC2}, though there is the important but subtle difference that here we are controlling the gradient flow map, while in \cite{ChC2} the map in question was a projection map.  The reasoning behind this difference is that we will need to compare balls over large distances, and a projection map will break down over such distances while the gradient flow will not.

\begin{lemma}\label{l:distseperation}
Let $\sigma_1,\sigma_2$ be two unit speed geodesics in $M$ and let $h:M\rightarrow \RR$ be a smooth function.  Then the following estimate holds:
\begin{align}
\left|\frac{d}{dt}d(\sigma_1(t),\sigma_2(t))\right| \leq |\nabla h-\sigma_1'|(\sigma_1(t))+|\nabla h-\sigma_2'|(\sigma_2(t)) + \inf\int_{\gamma_{\sigma_1(t),\sigma_2(t)}}|\Hess_h|
\end{align}
where $\inf$ is taken with respect to all minimizing geodesics connecting $\sigma_1(t)$ to $\sigma_2(t)$, and the derivative is meant in the sense of forward difference quotients at
non-differentiable points.
\end{lemma}

\begin{proof}
First note that without loss we can assume we are estimating at $t=0$, and by an approximation argument we can assume that for every $s$ in a small neighborhood of $0$ that the geodesic from $\sigma_1(s)$ to $\sigma_2(s)$ is unique.  We call these geodesics $\tau_s$ and let $l_s\equiv d(\sigma_1(s),\sigma_2(s))$ be their lengths.  Now we have the following computation:
\begin{align}
 \int_0^t\int_0^{l_s}l_s\,\Hess_h(\dot \tau_s,\dot\tau_s)(\tau_s(v))dvds =& \int_0^tl_s\,\left(\langle\nabla h,\dot \tau_s\rangle(\tau_s(l_s))- \langle\nabla h,\dot \tau_s\rangle(\tau_s(0))\right)\notag\\
=& \int_0^tl_s\left(\langle\sigma_2',\dot \tau_s\rangle(\tau_s(l_s))- \langle\sigma_1',\dot \tau_s\rangle(\tau_s(0))\right)\\
&+ \int_0^tl_s\,\left(\langle\nabla h-\sigma_1',\dot \tau_s\rangle(\tau_s(l_s))- \langle\nabla h-\sigma_2',\dot \tau_s\rangle(\tau_s(0))\right)\notag\\
= &\frac{1}{2}\left(l_t^2-l_0^2\right) + \int_0^tl_s\,\left(\langle\nabla h-\sigma_1',\dot \tau_s\rangle(\tau_s(l_s))- \langle\nabla h-\sigma_2',\dot \tau_s\rangle(\tau_s(0))\right)\notag\, .
\end{align}
Rearranging terms and dividing by $t$ gives
\begin{align}
\frac{1}{2t}\left[d^2(\sigma_1(t),\sigma_2(t))-d^2(\sigma_1(0),\sigma_2(0))\right] &\leq \frac{1}{t}\int_0^tl_s\,|\nabla h -\sigma_1'| + \frac{1}{t}\int_0^t l_s\,|\nabla h -\sigma_2'|\\
&+ \frac{1}{t}\int_0^t\int_0^{l_s}l_s\,|\Hess_h|(\tau_s(v))\,dv\,ds\, .\notag
\end{align}
Letting $t$ tend to zero and dividing by $d(\sigma_1(0),\sigma_2(0))$ gives the result.
\end{proof}

The next result is the primary use of the scaled segment inequality of \cite{ChC1} (see theorem 2.11 there).  This Lemma will be combined with the estimates of Theorem \ref{t:mainregthm1} in order to see that Lemma \ref{l:distseperation} can be applied to control the gradient flow map.

\begin{lemma}\label{l:seghessian}
Let $t\in (\delta,1-\delta)$, $0\leq s\leq t-\delta$ and let $c_s^t:B_r(\gamma(t))\times B_r(\gamma(t))$ be the characteristic function of the set $\mathcal{A}_s^t(r)\times\mathcal{A}_s^t(r)$.  Then we have the following:
\begin{align}
 \fint_{B_r(\gamma(t))\times B_r(\gamma(t))}c_s^t(x,y)\,\left(\int_{\gamma_{\psi_s(x),\psi_s(y)}}|\Hess_h|\right)\leq C(n,\delta)\,r\,\left(\frac{\Vol(B_r(\gamma(t-s)))}{\Vol(B_r(\gamma(t)))}\right)^{2}\fint_{B_{5r}(\gamma(t-s))}|\Hess_h|.
\end{align}
\end{lemma}

\begin{proof}
We begin with the computation
\begin{align}
\fint_{B_r(\gamma(t))\times B_r(\gamma(t))}c_s^t(x,y)\left(\int_{\gamma_{\psi_s(x),\psi_s(y)}}|\Hess_h|\right) &= \fint_{\mathcal{A}_s^t(r)\times\mathcal{A}_s^t(r)}\int_{\gamma_{\psi_s(x),\psi_s(y)}}|\Hess_h| \\
&\leq C(n,\delta)\fint_{\psi_s(\mathcal{A}_s^t(r))\times\psi_s(\mathcal{A}_s^t(r))}\int_{\gamma_{x,y}}|\Hess_h|\, ,\notag
\end{align}
where the last inequality follows from the volume comparison under the gradient flow.  Since $\psi_s(\mathcal{A}_s^t(r))\subseteq B_{2r}(\gamma(t-s))$ by definition we may apply the scaled segment inequality to get
\begin{align}
\int_{\psi_s(\mathcal{A}_s^t(r))\times\psi_s(\mathcal{A}_s^t(r))}\int_{\gamma_{x,y}}|\Hess_h| &\leq C(n)\,r\,\Vol(\psi_s(\mathcal{A}_s^t(r)))\int_{B_{5r}(\gamma(t-s))}|\Hess_h|\notag\\
&\leq C(n)\,r\,\Vol(B_{5r}(\gamma(t-s)))\int_{B_{5r}(\gamma(t-s))}|\Hess_h|\\
&\leq C(n)\,r\,\Vol(B_{r}(\gamma(t-s)))^2\fint_{B_{5r}(\gamma(t-s))}|\Hess_h|\, ,\notag
\end{align}
where the last inequalities follow from volume monotonicity.  Finally, by dividing out by $\Vol(B_{r}(\gamma(t)))^2$ and using volume comparison one more time we have our result.
\end{proof}

\subsection{Volume comparison}\label{ss:volcomparison}

We are now in a position to tackle the technical heart of the construction.  The goal of this Section is to prove the following Proposition, which gives at least some base control over the drifting of points under the gradient flow.  In particular, the next Proposition tells us that for {\it most} points $z\in B_r(\gamma(t))$ that the minimizing geodesic $\gamma_{p,z}$ remains near $\gamma$ for a definite amount of time.

\begin{proposition}\label{p:volratio1}
There exists $r_0(n,\delta)$ and $\epsilon(n,\delta)$ such that if $\delta<t'<t<1-\delta$ with $|t-t'|\leq \epsilon$ then $\forall r\leq r_0$ we have
\begin{align}
 \frac{1}{2}\leq \frac{\Vol(\mathcal{A}^{t}_{t'}(r))}{\Vol(B_r(\gamma(t)))} \leq 2\, .
\end{align}
\end{proposition}

We will need an improvement on this in the proof of Theorem \ref{t:mainregthm1}, namely that this volume ratio is behaving in a H\"older fashion, but this alone has at least one useful consequence we will quickly discuss.  Notice that
$$
\Vol(B_r(\gamma(t-t')))\geq C(n)\Vol(\psi_{t-t'}(\mathcal{A}^t_{t'}(r)))\geq C(n)\Vol(B_r(\gamma(t)))\, ,
$$
and that by applying Proposition \ref{p:volratio1} to the geodesic $\bar\gamma(t)\equiv \gamma(d_{p,q}-t)$ we obtain the reverse inequality
$$
\Vol(B_r(\gamma(t)))\geq C(n)\Vol(B_r(\gamma(t-t')))\, ,
$$
for $|t-t'|\leq\epsilon(n,\delta)$.  Iterating this immediately gives us:

\begin{corollary}
There exists $r_0(n,\delta)$ and $C(n,\delta)$ such that for all $s,t\in (\delta,1-\delta)$ and for any $r\leq r_0$ we have that
\begin{align}
C^{-1}\leq \frac{\Vol(B_r(\gamma(s)))}{\Vol(B_r(\gamma(t)))} \leq C\, .
\end{align}
\end{corollary}

This gives the interesting result that two points in the interior of a limit geodesic are absolutely continuous with respect to the renormalized limit measure relative to one another.  There is, in fact, a stronger version of this we will get to shortly. First we finish the proposition.

\begin{proof}[Proof of Proposition \ref{p:volratio1}]
Let us fix $t\in (\delta,1-\delta)$ and define
\begin{align}
S_t\equiv\{s\in(\delta,1-\delta): \frac{1}{2}< \frac{\Vol(B_r(\gamma(s)))}{\Vol (B_r(\gamma(t)))} < 2\, \forall r\leq r_0\} \, ,
\end{align}
where $r_0\leq \bar\epsilon(n,\delta)$, where $\bar\epsilon(n,\delta)$ is from Theorem \ref{t:mainregthm1}.  We will first claim that there is an $\epsilon(n,\delta)$ such that $[t-\epsilon,t+\epsilon]\subseteq S_t$, which notice is a strictly weaker claim than that of the proposition.

Notice first that since $M$ is a smooth manifold that for all $r$ sufficiently small (depending on $M$) that
$$\frac{\Vol(B_r(\gamma(s)))}{w_nr^n}$$ 
is uniformly close to one for every $s$.  In particular, it is easy to see that $S_t$ is an open set.  We will find $\epsilon(n,\delta)$ such that $[t-\epsilon,t+\epsilon]\cap S_t$ is closed, and then the claim will follow.

To do this we begin by finding the relevant estimates, these will make heavy use of Theorem \ref{t:mainregthm1} and Lemma \ref{l:geo_gradh}.  So let $\epsilon>0$ not yet be specified and $t'\in\bar S_t\cap [t-\epsilon,t+\epsilon]$, with either $|t'-t|= \epsilon$ or with $t'$ being the closest point of $\bar S_t\setminus S_t$ to $t$, where $\bar S_t$ is the closure of $S_t$.  Note that $t'\neq t$ by openness.  We of course wish to show $t'\equiv t-\epsilon$ for $\epsilon$ effectively chosen.  We can assume without loss of generality that $t'<t$ and get that
\begin{align}\label{e:VR}
\frac{1}{2}\leq \frac{\Vol(B_r(\gamma(s)))}{\Vol (B_r(\gamma(t)))} \leq 2\, \forall s\in [t',t] \text{ and } \forall r\leq r_0 \, .
\end{align}

Now recall the excess function $e_{p,q}(x)\equiv d(p,x)+d(x,q)-d(p,q)$ and let
\begin{align}
I^r_s\equiv \fint_{B_r(\gamma(t))\times B_r(\gamma(t))}\int_0^s c_u^t(x,y)\left(\int_{\gamma_{\psi_u(x),\psi_u(y)}}|\Hess_{h_{r^2}}|\right)\,du\,dv_g(x)\,dv_g(y)\, ,
\end{align}
where $h_{r^2}$ is the parabolic approximation function from Subsection \ref{s:parappxfun} and $c^t_u$ is the characteristic function $\mathcal{A}^t_u(r)\times\mathcal{A}^t_u(r)$.  Let us define
\begin{align}
T^r_{\eta}\equiv \left\{x\in B_r(\gamma(t)): e_{p,q}(x)\leq \eta^{-1} r^2 \text{ and } \fint_{\{x\}\times B_r(\gamma(t))}\int_0^{t-t'} c_s^t(x,y)\left(\int_{\gamma_{\psi_s(x),\psi_s(y)}}|\Hess_{h_{r^2}}|\right)\leq \eta^{-1}I^r_{t-t'}\right\}\, ,
\end{align}
and with $x\in T^r_\eta$ let us define
\begin{align}
T^r_\eta(x) \equiv \left\{y\in B_r(\gamma(t)): \int_0^{t-t'} c_s^t(x,y)\left(\int_{\gamma_{\psi_s(x),\psi_s(y)}}|\Hess_{h_{r^2}}|\right)ds\leq \eta^{-2}I^r_{t-t'}\right\}\, .
\end{align}

For the proof of the claim we will end up picking $\eta$ some fixed small constant, though because we will need it later we will be very explicitly about the dependence of $\epsilon$ on the choice of $\eta$.  Note from the integral excess inequality Theorem \ref{t:L1_excess} that
\begin{align}\label{e:volratio2}
\frac{\Vol(T^r_\eta)}{\Vol(B_r(\gamma(t)))}\geq 1-C(n,\delta)\eta\, ,
\end{align}
and hence
\begin{align}
\frac{\Vol(T^r_\eta(x))}{\Vol(B_r(\gamma(t)))}\geq 1-C(n,\delta)\eta\text{ }\forall x\in T^r_\eta\, .
\end{align}
Note also from Lemma \ref{l:seghessian}, Theorem \ref{t:mainregthm1} and (\ref{e:VR}) that
\begin{align}\label{e:I}
I^r_{t-t'}&=\int_0^{t-t'}\fint_{B_r(\gamma(t))\times B_r(\gamma(t))} c_u^t(x,y) \left(\int_{\gamma_{\psi_u(x),\psi_u(y)}} |\Hess_{h_{r^2}}|\right)\, \notag\\
&\leq C(n,\delta)\,r\int_0^{t-t'}\left(\frac{\Vol(B_r(\gamma(t-u)))}{\Vol (B_r(\gamma(t)))}\right)^2 \fint_{B_{5r}(\gamma(t-s))}|\Hess_{h_{r^2}}|\notag\\
&\leq C(n,\delta)\,r \int_0^{t-t'}\fint_{B_{5r}(\gamma(t-s))}|\Hess_{h_{r^2}}|\notag\\
&\leq C(n,\delta)r \sqrt{t-t'}\left(\int_\delta^{1-\delta} \fint_{B_{5r}(\gamma(s))}|\Hess_{h_{r^2}}|^2\right)^{1/2}\notag\\
&\leq C(n,\delta)\,\sqrt{t-t'}\,r\, .
\end{align}
It follows from Lemma \ref{l:geo_gradh} that if $x\in T^r_\eta$ and $y\in T^r_\eta(x)$, then for unit speed minimal geodesics $\sigma_x$ from $p$ to $x$ and $\tau_s$ from $\psi_s(x)$ to $\psi_s(y)$ we have
\begin{equation}\label{e:eta1}
 \int_{t'}^t|\nabla h_{r^2}-\nabla d_p|\leq \eta^{-1/2}\,C(n,\delta)\,\sqrt{t-t'}\,r\, ,
\end{equation}
and
\begin{equation}\label{e:eta2}
\int_{t'}^t\int_{\tau_s}c_s^t(x,y)\,|\Hess_{h_{r^2}}|\leq \eta^{-2}\,C(n,\delta)\,\sqrt{t-t'}\,r\, .
\end{equation}

Now let us give an imprecise outline of how the proof of the claim will proceed.  We wish to estimate the volume of the set of points $z\in B_r(\gamma(t))$ for which $\gamma_{p,z}(u)$ remains near $\gamma(u)$ for all $t'\leq u\leq t$.  Volume monotonicity tells us that if this set is large, relative to $\Vol(B_r(\gamma(t)))$, then the volume of $B_r(\gamma(t'))$ is bounded below by the volume of $B_r(\gamma(t))$.  The argument will be symmetric in $t$ and $t'$, and thus we will be able estimate the points  $z\in B_r(\gamma(t'))$ for which the geodesics $\gamma_{q,z}$ remain near $\gamma$, and hence also bound $\Vol(B_{r}(\gamma(t))$ from below by $\Vol(B_r(\gamma(t')))$.

To simplify matters for our outline, let us assume briefly that $\gamma(t)\in T^r_\eta$.  Then for any $x\in T^r_\eta(\gamma(t))\cap T^r_\eta$ we may use (\ref{e:eta1}) and (\ref{e:eta2}), along with Lemma \ref{l:distseperation}, so that we will be able to conclude that
$$
d(\gamma_{p,x}(t'),\gamma(t'))<C\eta^{-2}\sqrt{t-t'}r\leq C\eta^{-2}\sqrt{\epsilon}r\, .
$$
In particular, the minimizing geodesics between $\phi_{t-u}(x)=\gamma_{p,x}(u)$ and $\gamma(u)$ cannot grow in length too quickly, and by fixing $\eta>0$ and $\epsilon(n,\delta)>0$ correspondingly small we have the desired conclusion of the last paragraph.

The primary issue with this outline is that there is no reason we can assume $\gamma(t)\in T^r_\eta$.  Instead, we will connect the points $x\in T^r_\eta$ to $\gamma(t)$ by a piecewise geodesic whose length is not much larger than $r$.  The vertices of this piecewise geodesic will be denoted by $\{x_i\}$, with $x_0=x$, and will satisfy
\begin{align}
 x_{i+1}\in T^{r_i}_\eta(x_i)\cap T^{r_{i+1}}_\eta\, .
\end{align}
It will turn out that this is enough to show that the piecewise geodesics with vertices defined by $\{\psi_{t-u}(x_i)\}$ will also have length roughly equal to $r$, which in particular shows the desired conclusion that $\psi_{u-t}=\gamma_{p,x}(u)$ does not stray too far from $\gamma(u)$.  Now let us proceed to make this all rigorous.

Let $0<\mu(n,\delta,\eta)<\frac{1}{10}$ be chosen momentarily with $r_i\equiv \mu^{i}r$.  Let 
\begin{align}
x\equiv x_0\in T^r_\eta\, 
\end{align}
be arbitrary and let us define $x_i$ inductively in two steps as follows.  First, given $x_i\in T^{r_i}_\eta$ let 
\begin{align}
x_{i+1}\in T^{r_i}_\eta(x_i)\cap T^{r_{i+1}}_\eta\, .
\end{align}
Note that by a simple volume comparison argument using (\ref{e:volratio2}) that if we choose $\mu\equiv \mu(n,\delta)\eta^{\frac{1}{n}}$, with $\mu(n,\delta)$ sufficiently small, then for all $\eta\leq \eta_0(n,\delta)$ sufficiently small the sets $T^{r_{i+1}}_\eta$ and $T^{r_i}_\eta(x_i)$ will have nonempty intersection by their almost maximal volume properties, and hence such a $x_{i+1}$ will always exist.  Now we wish to end this induction after a finite number of steps with a specially chosen last $x_I$.  The claim is that for all $I$ large enough it automatically holds that we can pick the vertex $x_I$ with the property that
\begin{align}
d(\psi_s(x_I),\gamma(t-s))\leq (1+\frac{\mu}{10})r_{I}\, ,
\end{align}
for all $s\leq \epsilon(n,\delta,\eta)$.  We should note that {\it apriori} we make and need no claims about effective control over how large $I$ has to be chosen, only that there exists such an $I$.  To see that such an $I$ exists is where the Jacobi estimate of Lemma \ref{l:jacobiest} come in.  So let us define
\begin{align}
 H_{r}\equiv \{y\in B_r(\gamma(t)): d(\psi_s(y),\gamma(t-s))
 \leq (1+2C(n,\delta)\sqrt{s})\,r\, \text{  }\,\forall s\leq t-\delta\}\, ,
\end{align}
where $C(n,\delta)$ in the definition is chosen to be twice the constant from Lemma \ref{l:jacobiest}.  Because $\psi_s$ is a smooth map in a neighborhood of $\gamma([\delta,1-\delta])$, and because Jacobi fields satisfy the estimates of Lemma \ref{l:jacobiest}, we see that 
\begin{align}
\lim_{r\to 0}\frac{\Vol(H_r)}{\Vol(B_r(\gamma(t)))}=1\, .
\end{align}
In particular, there exists $\epsilon(n,\delta,\eta)\equiv \epsilon(n,\delta)\eta^{\frac{1}{2n}}$ such that for $I$ sufficiently large we may pick 
$$
x_I\in T^{r_{I-1}}_\eta(x_{I-1})\cap H_r\, ,$$ 
and hence
$$
d(\psi_s(x_I),\gamma(t-s))\leq (1+\frac{\mu}{10})r_{I}\, ,
$$ 
for all $s\leq \epsilon(n,\delta,\eta)$ as claimed.  Note that although $I$ depends on the manifold and geodesic in question the constant $\epsilon(n,\delta,\eta)$ does not.

Now let $\sigma(s)$ be the piecewise geodesic with vertices $\{\psi_s(x_i)\}_0^I$, and let $\sigma_i(s)$ be the segments connecting $\psi_s(x_i)$ to $\psi_s(x_{i+1})$.  Assume $i$ is such that 
\begin{align}
\psi_s(x_{i+1})\in B_{(1+\mu)r_{i+1}}(\gamma(t-s))\, ,
\end{align}
for all $s\leq t-t'$ and let 
\begin{align}
s_i\equiv \min\{t-t',\sup\{u:\psi_s(x_i)\in B_{(1+\mu)r_{i}}(\gamma(t-s)) \forall s\leq u\}\}\, .
\end{align}
So $s_i$ is the maximum $s$, up to $t-t'$, such that $\psi_s(x_i)$ remains in $B_{(1+\mu)r_i}(\gamma(t-s))$.  Now for any such $i$ as in our assumption and all $s\leq s_i$ we have that the characteristic function $c_t^s(x_i,x_{i+1})$ is identically one, and hence by Lemma \ref{l:distseperation} and equations (\ref{e:eta1}),(\ref{e:eta2}) we have that
\begin{align}
||\sigma_i(s_i)|-|\sigma_i(0)||\leq C(n,\delta)\,\eta^{-2}\,\sqrt{t-t'}\,r_i\leq \frac{\mu}{10}\, r_i\, ,
\end{align}
where the last inequality holds so long as $|t-t'|\leq \epsilon(n,\delta,\eta)\equiv\epsilon(n,\delta)\eta^{4+\frac{1}{2n}}$ are chosen sufficiently small.  In particular, we see that 
$$
|\sigma_i(s_i)|<(1+\frac{1}{2}\mu)r_i\, ,
$$
and hence 
$$
\psi_{s_i}(x_i)\in B_{(1+\mu)r_i}(\gamma(t-s_i))\, ,
$$ 
and thus we have that $s_i\equiv t-t'$.  Therefore, we have shown that for all $i$ such that 
$$
\psi_s(x_{i+1})\in B_{(1+\mu)r_{i+1}}(\gamma(t-s))\, \forall s\leq t-t'\, ,$$ 
we have that 
$$
\psi_s(x_{i})\in B_{(1+\mu)r_{i}}(\gamma(t-s))\, \forall s\leq t-t'\, .$$  
In particular, since this holds for $i=I-1$ it holds for all $i$ and hence we have that for all $\eta\leq\eta_0(n,\delta)$ there exists $\mu(n,\delta,\eta)=\mu(n,\delta)\eta^{\frac{1}{n}}$ and $\epsilon(n,\delta,\eta) =\epsilon(n,\delta)\eta^{4+\frac{1}{2n}}$ such that if $x\in T^r_\eta$ then $\psi_s(x)\in B_{(1+\mu)r}(\gamma(t-s))$ for all $s\leq t-t'$.  This, in particular, implies that 
\begin{align}
T^r_\eta\subseteq \mathcal{A}_t^s\,\forall s\leq t-t'\, .
\end{align}
We are nearly done with the claim.  To finish it note that this implies that
\begin{align}
\frac{\Vol(B_{r}(\gamma(t-t')))}{\Vol(B_{r}(\gamma(t)))}&\geq \frac{1}{(1+C(n)\mu)^n}\frac{\Vol(B_{(1+\mu)r}(\gamma(t-t')))}{\Vol(B_{r}(\gamma(t)))}\geq \frac{1}{(1+C(n)\mu)^n}\frac{\Vol(\psi_{t'}(T^r_\eta))}{\Vol(B_{r}(\gamma(t)))}\notag\\
&\geq \frac{1}{(1+C(n)\mu)^n(1+C(n)\epsilon)^n}\frac{\Vol(T^r_\eta)}{\Vol(B_{r}(\gamma(t)))}\geq \frac{1-C\eta}{(1+C(n)\mu)^n(1+C(n)\epsilon)^n}\, .
\end{align}
Hence, for $\eta(n,\delta)$ sufficiently small we have that $$
\frac{\Vol(B_{r}(\gamma(t-t')))}{\Vol(B_{r}(\gamma(t)))}>\frac{1}{2}\, .$$  
To see the reverse inequality we argue in a verbatim manner with respect to the gradient flow by the function $-\nabla d_q$, which shows that $t'\in S_t$ and hence $t'=t-\epsilon(n,\delta)\eta^{4+\frac{1}{2n}}$, which proves the claim.  The proof of the proposition follows immediately because with $t'=t-\epsilon$ we see that $T^r_\eta\subseteq \mathcal{A}_t^s$ $\forall s\leq \epsilon$.
\end{proof}

\subsection{Proof of Theorem \ref{t:holder}}\label{ss:proofmaintheorem}

We can now finish the proof of Theorem \ref{t:holder}:

\begin{proof}[Proof of Theorem \ref{t:holder}]
Let us begin by summing up some of the technical constructions obtained in the proof of Proposition \ref{p:volratio1}.  It was shown that for every $\eta\leq \eta_0(n,\delta)$ and $r\leq r_0(n,\delta)$ that there exists $\mu\equiv \mu(n,\delta)$ and $\epsilon\equiv \epsilon(n,\delta)$ such that if $x\in T^r_\eta$ and $y\in T^r_\eta\cap T^r_\eta(x)$ then the following hold
\begin{align}\label{e:distcompare}
\psi_s(x)\in B_{(1+\mu\eta^{\frac{1}{n}})r}(\gamma(t-s))\,\, \forall s\leq\epsilon\eta^{2\frac{1+2n}{n}}\\
|d(\psi_s(x),\psi_s(y))-d(x,y)|\leq \mu\eta^{\frac{1}{n}}r\,\, \forall s\leq\epsilon\eta^{2\frac{1+2n}{n}}\, ,
\end{align}
with the additional property that
\begin{align}
 \frac{\Vol(T^r_\eta)}{\Vol(B_r(\gamma(t)))},\text{  }\frac{\Vol(T^r_\eta(x))}{\Vol(B_r(\gamma(t)))}\geq 1-C(n,\delta)\,\eta\, .
\end{align}
Given this we see that $T^r_\eta$ is an $C(n,\delta)\eta^{\frac{1}{n}}$ dense subset.  Further for $s\leq\epsilon\eta^{2\frac{1+2n}{n}}$ we see that
\begin{align}
\Vol(B_{r}(\gamma(t-s)))&\geq (1-C(n,\delta)\eta)\Vol(B_{(1+\mu\eta^{\frac{1}{n}})r}(\gamma(t-s))) \geq (1-C(n,\delta)\eta)\Vol(\psi_s(T^\eta_r))\\
&\geq (1-C(n,\delta)\eta)(1-C(n,\delta)\eta^{4n+2})\Vol(T^\eta_r)\geq (1-C\eta)\Vol(B_r(\gamma(t)))
\, ,\notag
\end{align}
while we can get the opposite inequality by considering the flow by $-\nabla d_q$ and hence we get for $s\leq \epsilon\eta^{2\frac{1+2n}{n}}$
\begin{align}\label{e:VolHolder}
1-C\eta\leq\frac{\Vol(B_{r}(\gamma(t)))}{\Vol(B_{r}(\gamma(t-s)))}\leq 1+C\,\eta\, .
\end{align}
It follows from the above that 
$$
\frac{\Vol(\psi_s(T^r_\eta))}{\Vol(B_{(1+\mu\eta^{\frac{1}{n}})r}(\gamma(t-s)))}\geq 1-C\eta\, 
$$ 
and in particular that $\psi_s(T^r_\eta)$ is $C(n,\delta)\eta^{\frac{1}{n}}$ dense in $B_r(\gamma(t-s))$.  Given all this now let $x,y\in T^r_\eta$ be arbitrary points.  The volume constraints on $T^r_\eta$, $T^r_\eta(x)$ and $T^r_\eta(y)$ guarantee that there exists a point 
$$
z\in T^r_\eta\cap T^r_\eta(x)\cap T^r_\eta(y)\cap B_{C(n,\delta)\eta^{\frac{1}{n}}}(x)\, .$$  
It then follows from equation (\ref{e:distcompare}) that for $s\leq \epsilon\eta^{2\frac{1+2n}{n}}$ we have
\begin{align}
|d(\psi_s(x),\psi_s(y))-d(x,y)|&\leq |d(\psi_s(z),\psi_s(y))-d(z,y)|+|d(\psi_s(x),\psi_s(y))-d(\psi_s(z),\psi_s(y))|+|d(x,y)-d(x,z)|\notag\\
&\leq |d(\psi_s(z),\psi_s(y))-d(z,y)|+|d(\psi_s(x),\psi_s(z))|+|d(x,z)|\leq C\,\eta^{\frac{1}{n}}\, .
\end{align}
Rearranging and letting $\eta=\epsilon^{-n}s^{\frac{n}{2(1+2n)}}$ we see that $d_{GH}(B_{r}(\gamma(t)),B_{r}(\gamma(t-s)))\leq C(n,\delta)\,s^{\frac{1}{2(1+2n)}}$, as claimed.
\end{proof}

We quickly note the following corollary of equation (\ref{e:VolHolder}):

\begin{corollary}
We have for $s,t\in [\delta,1-\delta]$ and all $r\leq r_0(n,\delta)$ that
\begin{align}
\left|\frac{\Vol(B_r(\gamma(t)))}{\Vol(B_r(\gamma(t-s)))}-1\right| \leq C(n,\delta)\,|t-s|^{\frac{n}{2(1+2n)}}\, .
\end{align}
\end{corollary}

\section{Examples}\label{s:ExamplesII}

In this section we construct new examples of limits of Riemannian manifolds $(M^n_i,g_i,p_i)\rightarrow (X,d_X,p)$ that satisfy $\Ric_i\geq 0$ and a non-collapsing assumption $\Vol(B_1(p_i))\geq v>0$.  These examples are specifically meant to show the sharpness of the theorems of this paper and will illustrate what can happen along the interior of a minimizing geodesic in a limit space, in fact a limit minimizing geodesic.  Specifically Example \ref{ss:Example4} will exhibit a limit space $X$ with a limit minimizing geodesic $\gamma$ such that tangent cones from the same sequence of rescalings along $\gamma$ are not constant.  Example \ref{ss:Example5} will push this example further to show that for each $\delta>0$ there is a limit space such that the rate of change of these tangent cones is not $C^{1/2+\delta}$.  The constructions are based on multiply warped products and smoothing.

For Alexandrov spaces tangent cones are unique and Petrunin, \cite{Pn}, proved a conjecture of Yu. Burago  asserting that  tangent cones at any two points in the interior of a geodesic are isometric.  It is far too optimistic to think that such a result should hold for limit spaces with only lower Ricci bounds.  For instance, take a limit space $Y\times\RR$ where the tangent cone at $p\in Y$ is nonunique.  As in \cite{ChC2} one can even assume that this is a non-collapsed limit space.   If we consider the geodesic $\gamma\equiv\{p\}\times\RR$, then clearly the tangent cones at each point along the geodesic are not isometric.  However, what does hold is that tangent cones coming from the same sequence of rescalings are all unique (in fact we even have that for all $r>0$ and any $s<t$ that $B_r(\gamma(s))$ and $B_r(\gamma(t))$ are isometric).  One might conjecture that analogous to the Alexandrov case that tangent cones from the same sequence of rescalings are always unique, however Example \ref{ss:Example4} of this section shows that this is not the case.  Example \ref{ss:Example4} is a non-collapsed limit space $X$ with a minimizing geodesic $\gamma\subseteq X$ such that at each point of $\gamma$ the tangent cone is unique, but for any $s\neq t$ we have that the tangent cones at $\gamma(s)$ and $\gamma(t)$ are not isometric.  In particular, tangent cones from the same sequence of rescalings along $\gamma$ are not isometric, and any form of Burago's conjecture for limits with only lower Ricci bounds must fail.

Theorem \ref{t:tangentholder} gives us that tangent cones along the interior of a geodesic of a limit space change at most at a $C^{\alpha(n)}$ H\"{o}lder rate, and in fact an analysis of the proof shows that \textit{most} points change at a $C^{\frac{1}{2}}$ H\"older rate.  We would now like to see that these estimates are sharp.  In particular, let $X$ be a limit space and $\gamma:[a,b]\rightarrow X$ a unit speed limit minimizing geodesic with $r_i\rightarrow 0$ some fixed sequence such that the respective rescalings $(X,r^{-1}_id_X)$ at each $\gamma(t)$ converge to a limit tangent cone.  This gives us a well defined map $\gamma:[a,b]\rightarrow \mathcal{M}$, where $\mathcal{M}$ is the collection of compact metric spaces, by assigning to each $\gamma(s)$ the closed unit ball $\bar B_{1}(\gamma(s))$ in the tangent cone at $\gamma(s)$.  Theorem \ref{t:tangentholder} implies that when $\mathcal{M}$ is equipped with the Gromov-Hausdorff metric that this is a $C^{\alpha(n)}$ H\"{o}lder continuous map, and that for sets of large measure in each tangent cone there are in fact $C^{\frac{1}{2}}$ H\"older maps.  For each $\delta>0$ we construct in Example \ref{ss:Example5} a non-collapsed limit space $X_{\delta}$ with a limit\index{\footnote{}} minimizing geodesic $\gamma\subseteq X_{\delta}$ so that this induced map is not $C^{1/2+\delta}$.  Thus we will see that Theorem \ref{t:tangentholder} is sharp.

Topologically our examples of limit spaces are of the form of $C(S(M))$, that is the cone over the suspension of a smooth compact manifold $M$.  Generally speaking this will give rise to two singular rays, the cone rays through the suspension points of $S(M)$.  It is on these geodesic rays where we will construct limits with bad geodesic behavior.

\subsection{Example - non-constant tangent cones}\label{ss:Example4}
The purpose of this section is to construct a limit space $X$ with a limit minimizing geodesic $\gamma$ such that tangent cones coming from the same sequence of rescalings along $\gamma$ are not constant.  We begin by letting $M=\Sn^3$ be the three sphere, $g_0$ the round metric of constant curvature $1$ and $V_1,V_2,V_3$ a right invariant orthonormal basis.  For any numbers $\{m_1,m_2,m_3\}\in\RR$ we can consider the right invariant metric $g^{\Sn^3}$ on $\Sn^3$ defined by $\langle V_j,V_k \rangle_{g} = e^{2 m_j}\delta_{jk}$.  If $m_j(r,s)$ are smooth for $r\in (0,\infty)$ and $s\in (0,\pi)$, then we can define a metric on $C(S(\Sn^3))$ by
\begin{align}
g\equiv dr^2+a(r)^2\left(ds^2+b(s)^2g^{\Sn^3}(r,s)\right) \, ,
\end{align}
where $a(r)$ and $b(s)$ are any smooth positive warping factors which will be chosen later.  We will require two constraints on the functions $m_j(r)$.  First we require that
\begin{align}
\sum m_j(r,s) =const\, ,
\end{align}
be independent of $r$ and $s$.  This has the effect of fixing the volume element for each of the right invariant metrics on $\Sn^3$.  We also require that
\begin{align}
 \langle g',\dot g\rangle = 4\sum m'_j(r,s)\, \dot m_j(r,s) = 0\, .
\end{align}
This turns out to be a cross term in the Ricci curvature on $C(S(\Sn^3))$, which we want to vanish to show positivity.  Given these two conditions we have the following computation:

\begin{lemma}\label{l:RicciComputation}
 Let $(C(S(\Sn^3)),g)$ be a metric as above, then at any smooth point of $C(S(\Sn^3))$ we have the following:
\begin{enumerate}
 \item $\Ric_{rr} = -4\frac{a''}{a}-\sum(m'_k)^2$.
 \item $\Ric_{ss} = -3\frac{b^{\cdot\cdot}}{b}-\sum(\dot m_k)^2 + a^2\left[-\frac{a''}{a}-3\left(\frac{a'}{a}\right)^2\right]$.
 \item $\Ric_{jj} = \Ric^{\Sn^3}_{jj}+e^{2m_j}\left[b^2\left(-\frac{b^{\cdot\cdot}}{b}-2\left(\frac{\dot b}{b}\right)^2-3\frac{\dot b}{b}\dot m_j -m^{\cdot\cdot}_j\right)+a^2b^2\left(-\frac{a''}{a}-3\left(\frac{a'}{a}\right)^2-4\frac{a'}{a}m'_j-m''_j\right)\right]$.
 \item $\Ric_{rs}=\Ric_{rj}=\Ric_{sj}=\Ric_{jk}=0$.
\end{enumerate}
Here $\Ric^{\Sn^3}$ is the Ricci curvature on the three sphere with the induced right invariant metric $g^{\Sn^3}(r,s)$.
\end{lemma}

Now to make appropriate choices of the functions $a(r),b(s)$ we consider the following:
\begin{align}\label{e:a}
a(r) = \left\{ \begin{array}{rl}
  a_0r &\mbox{ for } r\leq t_0/2\\
  a_0r\left(1-\frac{a_1}{\log(-\log(r_0 r))}\right)&\mbox{ on } r\in [t_0,1]\\
  a_0r/2 &\mbox{ for } r\geq 2\\
  |a'|\leq 2a_0,\, a''<0 &\mbox{ on }r\in [t_0/2,2]\\
       \end{array} \right. \,,
\end{align}
\begin{align}\label{e:b}
b(s) = \left\{ \begin{array}{rl}
  \sin(s)&\mbox{ on } s\not\in [t_0/4,\pi-t_0/4]\\
  \sin(s)\left(1-\frac{b_1}{\log(-\log(s_0 \sin(s)))}\right)+b_0&\mbox{ on } s\in [t_0/2,\pi-t_0/2]\\
  |\dot b|\leq 2 ,\, b^{\cdot\cdot}\leq-b/2 &\mbox{ on }s\in [t_0/4,\pi-t_0/4]\\
  \end{array} \right. ,
\end{align}
where $a_0<1$ and $0<a_1,b_1,r_0,s_0,t_0$ are appropriately small constants that will be fixed and $b_0=o(t_0)$ is chosen below. To see the existence of such functions let us briefly consider $b(s)$, the construction is similar for $a(r)$.  In this case if we let $b_0\equiv \sin(\frac{t_0}{3})\left(\frac{b_1}{\log(-\log(s_0 \sin(\frac{t_0}{3})))}\right)$, then we can define
$$
\bar b(s)\equiv \min\{\sin(s), \sin(s)\left(1-\frac{b_1}{\log(-\log(s_0 \sin(s)))}\right)+b_0\}\, .
$$
We see that $\bar b$ satisfies all the requirements of $b$ away from $\frac{t_0}{3}$ and it satisfies the requirements globally in a distributional sense.  Hence, we can smoothen $\bar b$ near $\frac{t_0}{3}$ to construct the desired function $b$.  From these functions we have the following:

\begin{lemma}\label{l:ExampleConstr1}
There exist constants $0<a_0,a_1,b_1,r_0,s_0, m_0$, and $0<a_2, b_2$ such that for all $t_0$ sufficiently small if the $m_j(r,s)$ additionally satisfy
\begin{enumerate}
 \item $\Ric^{\Sn^3}(r,s)\geq 1$.
 \item $m_j\leq -m_0$.
 \item $|m_j'|\leq \frac{a_2\left(\frac{1}{r}\right)}{\log(-\log r_0 r)(-\log r_0 r)}$, $|m_j''|\leq a_2 r^{-2}$
 \item $|\dot m_j|\leq \frac{b_2\left(\frac{\cos s}{\sin s}\right)}{\log(-\log(s_0 \sin s))(-\log(s_0 \sin s))}$, $|m_j^{\cdot\cdot}|\leq b_2\sin^{-2} s$.
 \item $m_j' \equiv 0$ for $r\not\in [t_0,1]$ and $\dot m_j,m'_j \equiv 0$ for $s\not\in [t_0,\pi-t_0]$.
\end{enumerate}
Then the induced metric space $\left(C(S(\Sn^3)), dr^2+a(r)^2(ds^2+b(s)^2g_{\Sn^3}(r,s))\right)$ has nonnegative Ricci curvature at each smooth point.  Further, for $s\not\in [t_0,\pi-t_0]$ it is isometric to $dr^2+a^2(r)(S(\Sn^3,g_e))$, where $a$ is concave and $S(\Sn^3,g_e)$ represents the suspension over a small ellipse.
\end{lemma}

\begin{remark}
The key use of the above is that the conditions on $m_j'$ and $\dot m_j$ are nonintegrable.  This is crucial, in particular, for smoothing out possible limit spaces to actual smooth manifolds.  Because $a(r)$ is concave we can modify the metric $dr^2+a^2(r)(S(\Sn^3,g_e))$ into a smooth metric near the singular lines.
\end{remark}

\begin{proof}
We first observe the following computations for all $b_1,s_0,r_0,t_0$ sufficiently small:
\begin{align}
\frac{a''(r)}{a(r)} \leq \left\{ \begin{array}{rl}
  -\frac{a_1}{r^2\log(-\log(r_o r))^2(-\log(r_0 r))}&\mbox{ on } r\in [t_0,1]\\
  0 &\mbox{ on }r\not\in [t_0,1]\\
       \end{array} \right. ,
\end{align}
\begin{align}
\frac{b^{\cdot\cdot}(s)}{b(s)} \leq \left\{ \begin{array}{rl}
  -1 &\mbox{ on }s\not\in [t_0/4,\pi-t_0/4]\\
  -\frac{1}{2}-\left(\frac{\cos s}{\sin s}\right)^2\frac{b_1}{\log(-\log(s_0 \sin s))^2(-\log(s_0 \sin s))}&\mbox{ on } s\in [t_0/2,\pi-t_0/2]\\
  -\frac{1}{2} &\mbox{ on }s\in [t_0/4,\pi-t_0/4]\\
       \end{array} \right. .
\end{align}

The positivity of $\Ric_{rr}$ and $\Ric_{ss}$ is thus easy to check from Lemma \ref{l:RicciComputation}, the equations above and the conditions on $m_j$ with $a_2$ and $b_2$ sufficiently small relative to $a_1$ and $b_1$, respectively.  To check positivity of the $\Ric_{jj}$ term note the inequalities $|a^2 m_j''|\leq a_2$ and $|b^2 m_j^{\cdot\cdot}|\leq b_2$ as well as $|a'|\leq 2a_0$ and $|b^\cdot|\leq 2$.  Combining these with the first condition gives positivity for the Ricci curvature in the $\Sn^3$ directions.
\end{proof}

As an immediate consequence of the above we want to construct a non-collapsed limit space $(X,d)$ with a minimizing geodesic whose tangent cones coming from the same sequence of rescalings are not constant.  We pick our metric functions by the formula
$$
m_j(r,s)\equiv \psi_{t_0}(s)m_j(r) -\bar m_j\, ,$$
where $\psi_{t_0}(s)$ is a cutoff function which is $1$ in $[\frac{t_0}{2},\pi-\frac{t_0}{2}]$ with support in $[t_0,\pi-t_0]$, $m_j(r)$ are smooth functions of $r$ with support in $[\frac{1}{2},1]$ and the $\bar m_j$'s are constants.  Now recall the conditions
\begin{align}
\sum m_j(r,s)= \text{const}\, , \,  \sum m_j'(r,s)\dot m_j(r,s)=0\, ,
\end{align}
must be satisfied in order to apply Lemma \ref{l:RicciComputation}.  The first condition is equivalent to
\begin{align}
 \sum m_j(r)=0\, ,
\end{align}
for each $r$ and the second is equivalent to
\begin{align}
\sum m^2_j(r)\equiv const\, ,
\end{align}
being independent of $r$.  Hence, we apriori have that the $m_j(r)$'s may take values in a circle of possible values.

Now a quick computation tells us that $|m_j'(r,s)|\leq \sum |m'_j(r)|$ and $|\dot m_j(r,s)| \leq c|\dot\psi|$.  Let $m_j(r)$ be fixed and non-constant with support in $[\frac{1}{2},1]$ and let $\bar m_j\equiv 2m_0$ fixed with $m_0$ sufficiently large as in Lemma \ref{l:ExampleConstr1} $2)$, and $m_j(r)$ satisfying the estimates of Lemma \ref{l:ExampleConstr1} $3)$.  We can thus pick $a(r)$ and $b(s)$ as in Lemma \ref{l:ExampleConstr1} such that for all $t_0$ sufficiently small there is a cutoff function $\psi_{t_0}(s)$ so that the conditions of Lemma \ref{l:ExampleConstr1} are satisfied for the defined $m_j(r,s)\equiv \psi_{t_0}(s)m_j(r) - 2m_0$.  Note that the existence of such a $\psi_{t_0}$ follows because the condition on $|\dot m_j|$ is nonintegrable.  Thus for each $t_0>0$ sufficiently small we have a metric space
$$(C(S(\Sn^3)),dr^2+a_{t_0}(r)^2(ds^2+ b_{t_0}^2(s)\,g^{\Sn^3}(r,s)))\, ,$$
which has nonnegative Ricci curvature at each smooth point.  Note that $a_{t_0}(r)$ and $b_{t_0}(s)$ here actually depend on $t_0$, though only in a small neighborhood of the singular rays and in the term $b_0$, which decays faster than linearly in $t_0$.  Near the singular rays the metric space has a standard structure from Lemma \ref{l:ExampleConstr1} and because $a(r)$ is concave the metric can be smoothed out to even have positive sectional curvature near the singular ray.  In particular, we get a smooth Riemannian manifold with nonnegative Ricci curvature which is homeomorphic to $\RR^5$.

For each $i$ sufficiently large take $t_0\leq i^{-1}$ to produce a smooth space $(\RR^5,g_i)$ which is isometric to $(C(S(\Sn^3)),dr^2+a_i(r)^2(ds^2+ b_i^2(s)\,g^{\Sn^3}(r)))$ outside increasingly small neighborhoods of the singular rays, where $g(r)$ is the family of metrics on $\Sn^3$ defined by the metric functions $m_j(r,s)\equiv m_j(r)-2m_0$ and $t_0\equiv 0$.  That is, $g(r)$ represents the induced metric on $\Sn^3$ when $\psi(s)\equiv 1$ is taken to be identically one and $t_0$ is taken to be zero in equations (\ref{e:a}) and (\ref{e:b}) .  As $i\to\infty$ and hence $t_0\rightarrow 0$ we get that
$$(\RR^5,g_i)\stackrel{GH}{\rightarrow}(C(S(\Sn^3)),dr^2+a(r)^2(ds^2+ b^2(s)\,g^{\Sn^3}(r)))\, .$$
If $\gamma(r)$ is thus one of the singular rays in $C(S(\Sn^3))$, then for each $r$ we see that the tangent cone of the limiting metric space at $\gamma(r)$ is
$$(\RR\times C(\Sn^3),dt^2+ds^2+s^2\,g(r))\, .$$
In particular, we see that tangent cones from the same sequence of rescalings are changing along the geodesic, as claimed.

\subsection{Example - H\"older $\frac{1}{2}$ is sharp}\label{ss:Example5}

The purpose of the next example is to refine the previous construction so that the metrics $g^{\Sn^3}(r,s)$ are sufficiently irregular as to show that the H\"{o}lder continuity of Theorem \ref{t:tangentholder} is sharp.  In fact, for each $\delta>0$ we will construct a limit space such that along the interior of a minimizing geodesic there are tangent cones from the same sequence of rescalings which change at a $C^{\frac{1}{2}}$ H\"{o}lder rate, but not at a $C^{\frac{1}{2}+\delta}$ H\"{o}lder rate.  Another consequence of this example is that the hessian estimates from Theorem \ref{t:mainregthm1} are sharp.  More precisely the estimate $\int_{\gamma}\fint_{B_\epsilon(\gamma(r))}|\Hess_h|^2\leq C$ from Theorem \ref{t:mainregthm1} cannot be replaced with $\fint_{B_\epsilon(\gamma(r))}|\Hess_h|^2\leq C$ for each point $\gamma(r)$.

In the example we are interested only in the rate of change of the tangent cones along the interior of a limit geodesic.  Hence, we will only worry about constructing $g(r)$ in a neighborhood of $r=1$.  The rest of the space is much better behaved and it is not difficult to see how to smoothen out the construction on the rest of the space as in the previous example by cutting up the $a(r)$ function.  Now as in the previous subsection the example will be homeomorphic to $C(S(\Sn^3))$ equipped with a metric of the form
$$g\equiv dr^2+a(r)^2\left(ds^2+b(s)^2g^{\Sn^3}(r,s)\right)\, ,$$
where $g^{\Sn^3}(r,s)$ is a smooth two parameter family of metrics on $\Sn^3$ all defined by the relations $\langle V_j,V_k \rangle_{g} = e^{2 m_j}\delta_{jk}$ for a fixed right invariant basis $\{V_j\}$ which is orthonormal with respect to the standard metric.  The metric functions $m_j(r,s)$ are again assumed to satisfy the conditions that
\begin{align}
\sum m_j(r,s)= \text{const}\, , \,  \sum m_j'(r,s)\dot m_j(r,s)=0\, ,
\end{align}
so that Lemma \ref{l:RicciComputation} still holds.  As in the last example we define the function $b(s)$ by equation (\ref{e:b}), however we will define the function $a(r)$ by $a(r)\equiv 2-|r-1|^{1+\delta}$, at least in the neighborhood $[1-\delta,1+\delta]$.  Note that the simple estimates
\begin{align}\label{e:a2}
a<2,\,\, \frac{|a'|}{a}\leq |r-1|^{\delta},\,\, \frac{a''}{a}<-\frac{\delta}{|r-1|^{1-\delta}}
\end{align}
hold in $r\in (1-\delta,1)\cup (1,1+\delta)$ for all small $\delta>0$.  Now $a(r)$ is not a smooth function, which initially prohibits us from using it in the warped product construction, but the following observations take care of that.  We see that the estimate for $\frac{a'}{a}$ holds on all $(1-\delta,1+\delta)$ because $a$ is $C^1$, and the estimate on $\frac{a''}{a}$ holds distributionally on all of $(1-\delta,1+\delta)$.  Hence, we can smoothen $a(r)$ slightly to smooth functions which satisfy estimates (\ref{e:a2}) to as close of a degree as we like.  Using these functions in place of $a(r)$ in the below construction we can limit these smoothings and simply assume $a(r) = 2-|r-1|^{1+\delta}$.  Now by using Lemma \ref{l:RicciComputation} we have the following version of Lemma \ref{l:ExampleConstr1}:

\begin{lemma}\label{l:ExampleConstr2}
There exist constants $0<b_1,s_0, m_0$ and $0<a_2(\delta),b_2(\delta)$ such that for all $t_0$ and $\delta$ sufficiently small if the $m_j(r,s)$ additionally satisfy
\begin{enumerate}
 \item $m_j\leq -m_0$.
 \item $|m_j'|\leq a_2|r-1|^{-\frac{1-\delta}{2}}$, $|m_j''|\leq a_2 \sin^{-2}s$.
 \item $|\dot m_j|\leq \frac{b_2\left(\frac{\cos s}{\sin s}\right)}{\log(-\log(s_0 \sin s))(-\log(s_0 \sin s))}$, $|m_j^{\cdot\cdot}|\leq b_2\sin^{-2}s$.
 \item $\dot m_j, m'_j \equiv 0$ for $s\not\in [t_0,\pi-t_0]$.
\end{enumerate}
Then the induced metric space $(C(S(\Sn^3)), dr^2+a(r)^2\left(ds^2+b(s)^2g_{\Sn^3}(r,s)\right))$ has nonnegative Ricci curvature at each smooth point with $r\in[1-\delta,1+\delta]$.  Further for $s\not\in [t_0,\pi-t_0]$ it is isometric to $dr^2+a^2(r)(S(\Sn^3,g_e))$, where $a$ is concave and $S(\Sn^3,g_e)$ represents the suspension over a small ellipse.
\end{lemma}
\begin{remark}
 It is important in condition $2)$ that the bound on $m''$ be in terms of $s$ since the $a''$ term is not positive enough to control a $|r-1|^{-2}$ term, which is unlike the previous example and causes some complications in the construction.  Again we have that because $a(r)$ is concave we can modify the metric $dr^2+a^2(r)(S(\Sn^3,g_e))$ into a smooth metric near the singular lines.
\end{remark}

\begin{proof}
The proof is much the same as the computations for Lemma \ref{l:ExampleConstr1}, we point out the main observations required in the computation.  Note that using Lemma \ref{l:RicciComputation} we see that the requirement on $m_j'$ is precisely what is needed to guarentee that $\Ric_{rr}$ is positive.  We also see by using equation (\ref{e:a2}) that
\begin{align}
-\frac{a''}{a}-3\left(\frac{a'}{a}\right)^2-4|\frac{a'}{a}||m_j'|>0 \, ,
\end{align}
for $\delta$ small, so that $\Ric_{ss}$ is positive and most of the terms in $\Ric_{jj}$ are controlled.  The remaining obstacle is the $m_j''$ term from Lemma \ref{l:RicciComputation}, and this is precisely controlled by the assumption $|m_j''|\leq a_2 \sin^{-2}s$.  Hence the Ricci curvature is positive.
\end{proof}

To define the metric functions $m_j(r,s)$ let us begin by fixing $m_j(r):[1-\delta,1+\delta]\rightarrow \RR^3$ such that
\begin{align}
 \sum_j m_j(r)=0\, ,\, \sum_j m^2_j(r) = c\, ,
\end{align}
for some fixed constant $c$ and such that the map $m_j(r)$ is $C^{\frac{1+\delta}{2}}$ H\"{o}lder with
$$|m'_j(r)|\leq a_2|r-1|^{-\frac{1-\delta}{2}}\, ,$$
as in Lemma \ref{l:ExampleConstr2}.  The construction of the example will be slightly more complicated than Example \ref{ss:Example4}.  There we used one cutoff function in the definition of the metric functions $m_j(r,s)$, and its primary purpose was to make the metric one that we could be sure could be smoothed off.  However because $|m''_j(r)|\approx |r-1|^{-\frac{3}{2}}$ there is no hope to force positivity of the Ricci tensor if we continued in the manner of the last example.  To this end we define for each $i\in\NN$ the function $m_{ij}(r):[1-\delta:1+\delta]\rightarrow\RR^3$, which is a smooth approximation of $m_{j}(r)$.  We can easily construct such smoothings so that $m_{ij}(r)=m_j(r)$ outside smaller and smaller neighborhoods of $r=1$ with $\sum_j m_{ij}(r)=0$, $\sum_j m^2_{ij}(r)=c$ and such that $|m'_j(r)|\leq a_2|r-1|^{-\frac{1-\delta}{2}}$ with $|m_{ij}''(r)|\leq a_2 i^2$.  With that in place let $t_i\rightarrow 0$ be a decreasing sequence of positive numbers such that for each $i$ we can define the cutoff functions
$$
\psi_i(s) \equiv \left\{ \begin{array}{rl}
  1&\mbox{ for } s\in [(3t_i+t_{i+1})/4,(t_{i-1}+3t_i)/4] t_i\\
  0 &\mbox{ for } s\not\in [(t_i+t_{i+1})/2,(t_{i-1}+t_i)/2]\\
       \end{array} \right. ,
$$
and with the properties that
\begin{align}
|\dot \psi_i|\leq \frac{b_2\left(\frac{\cos s}{\sin s}\right)}{\log(-\log(s_0 \sin s))(-\log(s_0 \sin s))} \, ,
\end{align}
and
\begin{align}
|\psi_i^{\cdot\cdot}|\leq b_2\sin^{-2}s \, ,
\end{align}
as in Lemma \ref{l:ExampleConstr2}.  Note that such a condition is possible because the right hand sides of each inequality are nonintegrable, though this forces the ratios $\frac{t_{i+1}}{t_i}$ to be tending to zero.  Now let us consider the following metric functions:
\begin{align}
m_{Nj}(r,s)\equiv \sum^{N}_{i=1}\psi_i(s)m_{ij}(r)-\bar m\, ,
\end{align}
where $\bar m$ is a constant.  Note that the $m_{Nj}$ do satisfy the requirements of a metric function and that by construction we have for $c$ sufficiently small and $\bar m$ sufficiently large that the conditions of Lemma \ref{l:ExampleConstr2} are satisfied for each $N$.  Hence for each $N$ we now have a metric space $(C(S(\Sn^3)), g_N)$ which has positive Ricci curvature at each smooth point and can be smoothed near the singular lines to obtain a sequence of smooth manifolds $(\RR^5_N,\tilde g_N)$ with the property that
$$(\RR^5,\tilde g_N)\rightarrow (C(S(\Sn^3)), g_\infty)\, .$$
If $\gamma(r)$ is a singular ray in this limit space, then at the point $\gamma(r)$ if the sequence $t_i\rightarrow 0$, then the resulting tangent cone is isometric to
$$(\RR\times C(\Sn^3),dt^2+ds^2+s^2g(r)))\, .$$
Here $g(r)$ is the metric on $\Sn^3$ induced by the metric functions $m_j(r)$.  Hence in a neighborhood of $r=1$ we have the metric cones are changing at a $C^{\frac{1+\delta}{2}}$ H\"{o}lder rate and not at a $C^{\frac{1}{2}+\delta}$ rate.  This constructs the desired limit space.

\appendix

\section{Extending geodesics}\label{s:extendgeodesic}

This section is dedicated to proving a technical lemma.  Recall that on a smooth Riemannian manifold $(M,g,p)$ that $a.e.$ pair of points $(x,y)\in M\times M$ lie in the \textit{interior} of a minimizing geodesic.  We wish to show that on a Ricci limit space $M_\infty$ that similarly $\nu\times\nu$ a.e. pair $(x,y)\in M_\infty$ lie in the interior of a limit minimizing geodesic $\gamma$.  We will in fact prove a more effective version of this.  This result can be thought of as a higher degree of freedom analogue that the cut locus set $\Cl_x$ of a point $x\in M_\infty$ has zero measure, a result proven in \cite{H}.  The key point for both results is to identify this critical point set in terms of excess functions, which are themselves much easier to estimate and control than geodesics when passing to limits.

We begin with a few definitions.  Recall for $x$, $y\in M_\infty$ we define the excess function
\begin{equation}
e_{x,y}(z) \equiv d(x,z)+d(z,y)-d(x,y)\, .
\end{equation}
Thus the excess is how much the triangle inequality fails being an equality.  Note that $e_{x,y}(z)=0$ iff $z$ lies on the interior of a minimizing geodesic connecting $x$ and $y$.  Similarly, for $(x,y)\in M_\infty\times M_\infty$ we define the following diagonal excess function by
\begin{equation}
e_{(z,w)}(x,y)\equiv \frac{1}{\sqrt{2}}d_{M_\infty}(x,y) + d_{M\times M}((x,y),(z,w)) - \frac{1}{\sqrt{2}}d_{M_\infty}(z,w)\, .
\end{equation}
Note that $\frac{1}{\sqrt{2}}d_{M_\infty}(x,y)$ is the distance of the point $(x,y)$ from the diagonal $M_\infty\subseteq M_\infty\times M_\infty$, and hence the reference to this as a diagonal excess function.  We similarly have that $e_{(z,w)}(x,y)=0$ iff $(x,y)$ lies on the interior of a minimizing geodesic from $(z,w)$ to the diagonal $M_\infty$.  We define the following cutlocus and effective cutlocus sets:

\begin{align}
 &\Cl(M_\infty) \equiv \{(x,y)\in M_\infty\times M_\infty:\, e_{(z,w)}(x,y)>0\, \forall\, (z,w)\neq (x,y)\}.\\
 &\Cl(M_\infty,r) \equiv \{(x,y)\in M_\infty\times M_\infty:\, e_{(z,w)}(x,y)>0\, \forall\, (z,w)\not\in B^{M\times M}_{\sqrt{2}r}(x,y)\}.\\
 &\Cl(M_\infty,r,\epsilon) \equiv \{(x,y)\in M_\infty\times M_\infty:\, e_{(z,w)}(x,y)
 \geq \epsilon^2\, \forall\, (z,w)\not\in B^{M\times M}_{\sqrt{2}r}(x,y)\}\, .
\end{align}

We see that a point $(x,y)$ is not in the set $\Cl(M_\infty)$ iff there is a point $(z,w)$ and a minimizing geodesic from $D$ to $(z,w)$ which contains $(x,y)$ as an interior point.  We will see from the below lemma that this implies that there is a geodesic in $M_\infty$ which contains both the points $x$ and $y$ as interior points.  On the other hand $(x,y)$ is not in $\Cl(M_\infty,r)$ iff there is a geodesic in $M_\infty$ containing $x$ and $y$ such that these points are at least distance $r$ from the boundary of the geodesic, a point which also follows from the below lemma.  Finally, the point $(x,y)$ is not in $\Cl(M_\infty,r,\epsilon)$ iff there is an $\epsilon$-geodesic with likewise properties.  We will see by the end of this section that geodesic can be replaced by limit geodesic for each of these statements.  Important for us is that the sets $\Cl(M_\infty,r,\epsilon)$ are compact.

\begin{lemma}\label{l:diaggeo}
Let $(x,y)\in M_\infty\times M_\infty$ and $(z,z)$ be a point of the diagonal closest to $(x,y)$.  Then if $(\gamma_1(t),\gamma_2(t))$ is a minimizing geodesic in $M_\infty\times M_\infty$ connecting $(x,y)$ to $(z,z)$ then the join curve $\gamma\equiv \gamma_1\cup\gamma_2$ in $M_\infty$ is a minimizing geodesic connecting $x$ to $y$.  Further we have that $z$ is the midpoint of this geodesic.
\end{lemma}
\begin{proof}
Assume this is not the case, then there is a curve $\sigma:[0,\overline{x,y}]\rightarrow M_\infty$ connecting $x$ and $y$ satisfying $|\sigma|<|\gamma|$.  Then if we consider the curve $(\sigma(t),\sigma(\overline{x,y}-t)):[0,\frac{1}{2}\overline{x,y}]\rightarrow M_\infty\times M_\infty$ then this curve connects $(x,y)$ to the diagonal and has length strictly less than that of $(\gamma_1,\gamma_2)$, which is a contradiction.

To see that $z$ is the midpoint we can just check the possibilities.  So let $(x,y)\in M_\infty\times M_\infty$ and let $\gamma:[0,\overline{x,y}]\rightarrow M_\infty$ be any minimizing geodesic between $x$ and $y$ with $z\in\gamma$ such that the point $(z,z)$ is a point on the diagonal closest to $(x,y)$.  Hence for some $s\in [0,1]$ we have that $z=\gamma(s\,\overline{x,y})$.  Since a minimizing geodesic in $M_\infty\times M_\infty$ projects to minimizing geodesics in each factor we must have that the minimizing geodesic from $(x,y)$ to $(z,z)$ is of the form $(\gamma(st),\gamma(\overline{x,y}-t(1-s)))$.  Hence if we compute the length as a function of $s$ we get $l(s)^2=(s^2+(1-s)^2)\overline{x,y}^2$.  It is easy to check this is minimized only for $s=\frac{1}{2}$.
\end{proof}

Now we begin with the following estimate, which should be seen as a generalization of certain estimates on exponential maps obtained in \cite{ChC2}.  We begin by proving the estimate on smooth manifolds $(M^n,g,p)$ with $\Ric\geq -(n-1)$, we will then subsequently see that the estimates hold on limit spaces.  In the below lemma we are using for a subset $S\subseteq M$ the partial annulus
$$A_{\delta,\delta^{-1}}(S)\equiv \{(x,y)\in M\times M:\, p_{D}(x,y)\in S,\, \delta\leq \frac{1}{\sqrt{2}}d(x,y)\leq\delta^{-1}\}\, ,$$
where $p_{D}$ is the projection map to the diagonal and we are identifying $S$ with its image in the diagonal.

\begin{lemma}
For each $0<\delta<1$, $R>0$ and any $\epsilon\geq 0$ there exists $C(n,\delta,R)$ such that for any $S\in M$ with $S\subseteq B_R(p)$ we have that $\Vol(\Cl(M,r,\epsilon)\cap A_{\delta,\delta^{-1}}(S))\leq C\, r\, \Vol(B_1(p))$.
\end{lemma}
\begin{proof}
First it is enough to prove the claim for $\Cl(M,r)$ since the constant involved is independent of $\epsilon$.  Note again that the distance function on $M\times M$ to the diagonal can be written $d_{D}(x,y)=\frac{1}{\sqrt{2}}d(x,y)$.  In particular the laplacian of this distance function on $M\times M$ satisfies
$$
\Delta d_{D}(x,y)\leq \frac{n-1}{d_{D}(x,y)}\, .$$
Let us define the tube
$$T_s(S)\equiv \{(x,y)\in M\times M:\, p_{D}(x,y)\in S,\, d_D(x,y)\leq s\}\, ,$$
and note then that $A_{s_0,s_1}(S) = T_{s_1}\setminus T_{s_0}$.  The estimate on the laplacian of $d_D$ then tells us that at any smooth point of $\partial T_s(S)$ that the mean curvature is uniformly bounded from above in terms of $s$.  To finish the proof we simply observe, as is in the case for the standard cut locus of a point, that the effective cutlocus $\Cl(M,r)$ intersects each minimal geodesic leaving the diagonal $D$ on a set of measure at most $r$. Thus if $\chi_{\Cl(M,r)}$ is the characteristic function of $\Cl(M,r)$ we have by a coarea formula and the mean curvature estimate that
\begin{align}
\Vol(\Cl(M,r)\cap A_{\delta,\delta^{-1}}(S)) &= \int^{\delta^{-1}}_{\delta}\left(\int_{\partial T_{s}(S)}\chi_{\Cl(M,r)}\right)ds\\
&\leq c(n,\delta)\, r\, \Vol(\partial T_{\delta}(S))\leq C(n,\delta)\, r\, \Vol(B_1(S))\notag\,\notag\\
&\leq C(n,\delta,R)r\Vol(B_1(p)) \notag,
\end{align}
as claimed.
\end{proof}

Now let us point out the following two stability properties of $\Cl(M,r,\epsilon)$.  To begin with if
\begin{align}
(M_i,g_i,p_i)\rightarrow (M_\infty,d_\infty,p_\infty) \, ,
\end{align}
then we can define
\begin{align}
 \Cl(M_i,r,\epsilon)\stackrel{GH}{\rightarrow} \Cl_\infty(r,\epsilon)\, .
\end{align}
Note first that for each $\eta>0$ that we have
\begin{align}
 \Cl(M_\infty,r-\eta,\epsilon+\eta)\subseteq \Cl_\infty(r,\epsilon)\subseteq \Cl(M_\infty,r,\epsilon)\, ,
\end{align}
this follows from the stability of the excess function under Gromov-Hausdorff limits.  We secondly have that
$$B_\eta(\Cl(M,r,\epsilon))\subseteq \Cl(M,r+\eta,\epsilon)\, ,$$
for each $\eta>0$.  Thus by using these observations and a covering argument we may limit in precisely the manner of \cite{H} to obtain the corresponding result in the limit space:

\begin{proposition}
For each $0<\delta<1$, $R>0$ and any $\epsilon>0$ there exists $C(n,\delta,R)$ such that for any $S\in M_\infty$ with $S\subseteq B_R(p_\infty)$ we have that
\begin{equation}
\nu(\Cl(M_\infty,r,\epsilon)\cap A_{\delta,\delta^{-1}}(S))\leq C\, r\, .
\end{equation}
\end{proposition}

Now if $(x,y)\in \Cl_\infty(r,\epsilon)$ then there exists a \textit{limit} minimizing geodesic $\gamma$ with $x$ and $y$ as interior points which are at least a distance $r$ from the boundary of $\gamma$.  Thus as a consequence of the results of this section and the previous stability properties of $\Cl(M_\infty,r,\epsilon)$ we can let $\epsilon\rightarrow 0$ to have the following.

\begin{corollary}  \label{c:extending}
The following statements hold for each $S\subseteq M_\infty$, $0<\delta<1$, $R>0$ and $r>0$:
\begin{enumerate}
 \item If $(x,y)\in \Cl(M_\infty,r)$ then there exists a limit minimizing geodesic $\gamma$ with $x$ and $y$ as interior points which are at least a distance $r$ from the boundary of $\gamma$.
 \item If $S\subseteq B_R(p_\infty)$ then $\nu(\Cl(M_\infty,r)\cap A_{\delta,\delta^{-1}}(S))\leq C(n,\delta,R)\, r\,$
 \item $\nu\times\nu$ $a.e.$ pair of points $(x,y)$ lie in the interior of some limit minimizing geodesic.
\end{enumerate}
\end{corollary}

\section{Reifenberg property for collapsed limits}\label{s:ReifenbergCollapsed}
For non-collapsed limits a key regularity of a neighborhood of the regular set come from a Reifenberg type property, see appendix 1 of \cite{ChC2}.  This property roughly say that on all scales the space is Gromov-Hausdorff close to Euclidean space.  It is shown in \cite{ChC2} that this implies that a neighborhood of the regular set for non-collapsed limits is a $C^{\alpha}$ manifold.

In the general, not necessary non-collapsed case, we have the following (uniform) Reifenberg property for
geodesics contained in the regular set:

\begin{theorem}  \label{t:unireif}
Suppose that $\gamma: [0,\ell]\to M_{\infty}$ is a limit geodesic whose interior consists of $k$-regular points.  Given
$\epsilon>0$, and $\ell>s_2>s_1>0$, there exists $r_0>0$ such that for all $r_0>r>0$ and all $s_2\geq s\geq s_1$
\begin{equation}
d_{GH}(B_r(\gamma (s)),B_r^{\RR^k}(0))<\epsilon\,r\, .
\end{equation}
\end{theorem}

\begin{proof}
By compactness of the closed interval $[s_1,s_2]$, the theorem would follow if we knew that for each $s\in [s_1,s_2]$, there exists a $\delta=\delta (s)>0$ and a $r_0=r_0 (s)>0$ such that for all $t\in (s-\delta,s+\delta)$ and all $r_0>r>0$
\begin{equation}
d_{GH}(B_r(\gamma (t)),B_r^{\RR^k}(0))<\epsilon\,r\, .
\end{equation}
However, this follow easily using that $\gamma (s)$ is a $k$-regular point combined with Theorem \ref{t:holder}.
\end{proof}

In the non-collapsed case when combined with the volume convergence theorem of \cite{C3} (cf. also theorem 5.9 of \cite{ChC2} and section 3 of \cite{ChC3}) it follows that if $\gamma$ is as in Theorem \ref{t:unireif}, then an entire neighborhood of $\gamma |[s_1,s_2]$ consists of almost regular points.  Or, to be precise, an entire neighborhood
consists of $(\epsilon,k)$-regular points in the sense of definition 0.6 of \cite{ChC2}).

A key difference between the collapsed and non-collapsed case is the following:
\begin{itemize}
\item By \cite{C3} (see also \cite{C1}, \cite{C2}), then in the non-collapsed case closeness in the Gromov-Hausdorff sense to $n$-dimensional Euclidean space is equivalent to that the volume is almost maximal.  By the Bishop-Gromov volume comparison theorem once the volume is almost maximal on one scale, then it is also almost maximal on all smaller scales and hence by \cite{C3} also Gromov-Hausdorff close to $n$-dimensional Euclidean space on all smaller scales.  This property in the non-collapsed case is where the Reifenberg property naturally occur; see appendix 1 in \cite{ChC2}.
\item Even though there is no such monotonicity in the collapsed case, then a key point in the collapsed case is that the
H\"older continuity of tangent cones can at some level replace this monotonicity as is illustrated in Theorem \ref{t:unireif}.
\end{itemize}

Theorem \ref{t:unireif} generalizes immediately to the situation where the tangent cones are unique and constant along the interior of a geodesic segment, as it played no role in the proof of this theorem that each point on the geodesic was $k$-regular.  The only thing that mattered was that the tangent cone is unique at each interior point and independent of the particular point.    Thus we have the following:

\begin{theorem}  \label{t:unireifY}
Suppose that $\gamma: [0,\ell]\to M_{\infty}$ is a limit geodesic and that at each interior point the tangent cone is unique and equal to a fixed pointed metric space $(Y,0)$ ($0$ is the `cone' tip).  Given
$\epsilon>0$, and $\ell>s_2>s_1>0$, there exists $r_0>0$ such that for all $r_0>r>0$ and all $s_2\geq s\geq s_1$
\begin{equation}
d_{GH}(B_r(\gamma (s)),B_r^{Y}(0))<\epsilon\,r\, .
\end{equation}
\end{theorem}

\vskip2mm
Note that for a $k$-regular point $y$ there is no specific requirement
on the rate of convergence as $r_i\to 0$ of the family of rescaled spaces
$(M_{\infty},y,r^{-1}_id_{\infty})$ to the tangent cone $\RR^k$. Equivalently, prior to rescaling, the convergence to $\RR^k$ takes place
at the rate $o(r)$. For $\alpha > 0$ a point $y$ is called $(k,\alpha)$ -regular\footnote{The $(k,\alpha)$-regular points were introduced in section 3 of \cite{ChC4} and should not be confused with the $(\epsilon,k)$-regular points mentioned earlier in this paper that was introduced in definition 0.6 in \cite{ChC2}.}, if on
sufficiently small balls $B_r(y)$ the convergence to $\RR^k$ takes place at the
rate $0(r^{1+\alpha})$. The set of $(k,\alpha)$-regular points is denoted $\cR_{k,\alpha}$.
In section 3 of \cite{ChC4} it was shown that $\nu (\cR_k\setminus \cR_{k,\alpha})=0$ for some $\alpha (n) > 0$ and that $\cR_{k,\alpha}$ is a countable union of sets, each of which is bi-Lipschitz to a subset of $\RR^k$.
Finally, in section 4 of \cite{ChC4} it was shown for limit spaces that on the
set $\cR_{k,\alpha}$ any of the renormalized limit measures and the Hausdorff
measure are mutually absolutely continuous. It follows that the
collection of all renormalized limit measures determines a unique
measure class.

\end{document}